\documentclass[a4paper,11pt]{amsart}

\usepackage{a4wide, amsmath, amsfonts, amssymb, mathrsfs, amsthm, bbm, color, stmaryrd, breqn}
\usepackage[hyperfootnotes=false]{hyperref} 

\allowdisplaybreaks[2]

\numberwithin{equation}{section}

\newtheorem{theorem}{Theorem}[section]
\newtheorem{lemma}[theorem]{Lemma}

\newtheorem{remark}[theorem]{Remark}

\theoremstyle{definition}
\newtheorem{definition}[theorem]{Definition}
\newtheorem{example}[theorem]{Example}

\newcommand{\dd}{\,\mathrm{d}}
\renewcommand{\d}{\mathrm{d}}
\newcommand{\R}{\mathbb{R}}
\newcommand{\N}{\mathbb{N}}
\newcommand{\1}{\mathbf{1}}
\newcommand{\X}{{\bf X}}

\newcommand{\var}{\text{-}\mathsf{var}}
\newcommand{\vertiii}[1]{{\left\vert\kern-0.25ex\left\vert\kern-0.25ex\left\vert #1 \right\vert\kern-0.25ex\right\vert\kern-0.25ex\right\vert}}
\newcommand{\Tc}{\mathcal{T}}

\title{A Sobolev rough path extension theorem via regularity structures}

\author[Liu]{Chong Liu}
\address{Chong Liu, ShanghaiTech University, China}
\email{liuchong@shanghaitech.edu.cn}

\author[Pr{\"o}mel]{David J. Pr{\"o}mel}
\address{David J. Pr{\"o}mel, Universit{\"a}t Mannheim, Germany}
\email{proemel@uni-mannheim.de}

\author[Teichmann]{Josef Teichmann}
\address{Josef Teichmann, Eidgen{\"o}ssische Technische Hochschule Z{\"u}rich, Switzerland}
\email{josef.teichmann@math.ethz.ch}

\date{\today}

\begin{document}

\begin{abstract}
  We show that every $\R^d$-valued Sobolev path with regularity~$\alpha$ and integrability~$p$ can be lifted to a Sobolev rough path provided $1/2 >\alpha > 1/p \vee 1/3$. The novelty of our approach is its use of ideas underlying Hairer's reconstruction theorem generalized to a framework allowing for Sobolev models and Sobolev modelled distributions. Moreover, we show that the corresponding lifting map is locally Lipschitz continuous with respect to the inhomogeneous Sobolev metric.
\end{abstract}

\maketitle




\section{Introduction}

The cornerstone of rough path theory is the concept of a rough path in the sense of Lyons~\cite{Lyons1998}. Unlike a classical path~$X$ from an interval~$[0,T]$ to the Euclidean space~$\R^d$, a rough path $\mathbf{X}=(X,\mathbb{X})$ contains the additional information~$\mathbb{X}$ representing, loosely speaking, the iterated integrals of the path~$X$ against itself. However, at least in general, there is no simple canonical way to ensure the existence of these iterated integrals. This has led to the fundamental question whether every $\R^d$-valued path~$X$ can be lifted to a rough path~$\X$ in the sense that the projection of~$\X$ onto the path-level\footnote{That is, the first coordinate projection $\X = (X, \mathbb{X}) \mapsto X$.} is~$X$.

A first affirmative and non-trivial\footnote{Of course, for a smooth path a rough path lift can be constructed by, e.g., Riemann--Stieltjes integration.} answer was given by Lyons and Victoir~\cite{Lyons2007a}, proving, in particular, that an $\R^d$-valued H{\"o}lder continuous path can always be lifted to a H{\"o}lder continuous weakly geometric rough path. Using a re-parameterization argument, this directly reveals an extension theorem in terms of $p$-variation. While proof of Lyons and Victoir is considered to be non-constructive, an explicit approach, based on so-called Fourier normal ordering, was obtained by Unterberger~\cite{Unterberger2010}. More recently, further constructive proofs of the Lyons--Victoir extension theorem were derived by Broux and Zambotti~\cite{Broux2021}, using a sewing lemma for low regularity, and by Tapia and Zambotti~\cite{Tapia2020}, using an explicit form of the Baker--Campbell--Hausdorff formula. Notice that in \cite{Tapia2020} a Lyons--Victoir extension theorem is provided allowing even for anisotropic H{\"o}lder continuous paths, i.e., allowing each component of the underlying path to have a different H{\"o}lder regularity. Moreover, in~\cite{Liu2018} a Lyons--Victoir extension theorem for Sobolev paths is proven, using a discrete characterization of (non-linear) fractional Sobolev spaces.

The present work is a companion paper of \cite{Liu2018} and continues the above line of research. We explore an approach based on Hairer's theory of regularity structures~\cite{Hairer2014}, which goes back to \cite{Friz2014}, and show that every path with Sobolev regularity $\alpha \in (1/3,1/2)$ and integrability~$p>1/\alpha$ can be lifted to a weakly geometric rough path possessing exactly the same Sobolev regularity. While the rough path lift of a H{\"o}lder continuous path is a known and fairly simple application of Hairer's reconstruction theorem (\cite[Theorem~3.10]{Hairer2014}), see \cite[Proposition~13.23]{Friz2014} or \cite{Brault2019}, lifting a Sobolev path lies outside the current framework of regularity structures and thus requires some serious additional effort. Indeed, we need to use a Sobolev topology on the space of modelled distributions, as introduced in \cite{Hairer2017} and \cite{Liu2016} (see also \cite{Hensel2020}) and additionally to generalize the definition of models from the originally required H{\"o}lder bounds to some more general Sobolev bounds. In other words, we cannot apply Hairer's reconstruction theorem directly and instead need to generalize the essential features of Hairer's reconstruction operator to our setting allowing for Sobolev models and Sobolev modelled distributions, see Remark~\ref{rmk:new bounds} for a more detailed discussion. The rough path lift essentially relying on Hairer's reconstruction operator constitutes an explicit construction of a rough path above a given $\R^d$-valued path. In contrast to the rough path lift obtained in the spirit of Lyons and Victoir, the corresponding lifting map turns out to be continuous, cf. \cite[Proposition~13.23]{Friz2014} or \cite{Brault2019}. Indeed, we prove that the corresponding lifting map is locally Lipschitz continuous with respect to the inhomogeneous Sobolev metric.

\smallskip
\noindent{\bf Organization of the paper:} In Section~\ref{sec:Sobolev rough path} we introduce the notion of a rough path and some basic definitions. In Section~\ref{sec:lifting} we construct the rough path lifts of Sobolev paths and show that the corresponding lifting map is locally Lipschitz continuous.

\section{Sobolev rough path and basic notation}\label{sec:Sobolev rough path}

We start by introducing the notion of Sobolev rough paths in Subsection~\ref{subsec:Sobolev rough path} and some basic definitions in Subsection~\ref{subsec:basic notation}. 

\subsection{Sobolev rough path}\label{subsec:Sobolev rough path}

Since we focus throughout the entire work on the Sobolev regularity $\alpha \in (1/3,1/2)$, we only present the definitions below in the necessary generality to deal with this regularity. A more general treatment of Sobolev rough paths can be found in \cite{Liu2018,Liu2021} and for more comprehensive introduction to rough path theory, see e.g. \cite{Lyons2007,Friz2010,Friz2014}.

We first recall the underlying algebra structure of a rough path, which can be conveniently described by the free nilpotent Lie group $G^2(\R^d)$. Let $\R^d$ be the Euclidean space with norm~$| \cdot |$ for $d\in \N$ and let $C^{1\var}([0,T];\mathbb{R}^d)$ be the space of all continuous functions $Z\colon [0,T]\to \R^d$ of finite variation. For a path $Z\in C^{1\var}([0,T];\mathbb{R}^d)$, its step-$2$ signature is defined by 
\begin{align*}
  S_2(Z)_{s,t}
  :=\bigg (1, \int_{s<u<t}\dd Z_u, \int_{s<u_1<u_2<t}\dd Z_{u_1} \otimes \d Z_{u_2} \bigg) 
  \in T^2(\mathbb{R}^d):= \bigoplus_{k=0}^2 \big(\mathbb{R}^d\big)^{\otimes k},
\end{align*}
where $\big(\mathbb{R}^d\big)^{\otimes n}$ denotes the $n$-tensor space of $\mathbb{R}^d$ with the convention $(\R^d)^{\otimes 0}:=\mathbb{R}$, cf. \cite[Definition~7.2]{Friz2010}. We equip $T^2(\R^d)$ with the standard addition $+$, tensor multiplication~$\otimes$ and scalar product, and denote by~$\pi_i$ the projection from $T^2(\mathbb{R}^d)$ onto the $i$-th level, for $i=0,1,2$. The corresponding space of all these lifted paths is the step-$2$ free nilpotent group (w.r.t. $\otimes$) 
\begin{equation*}
  G^2(\mathbb{R}^d) := \{S_2(Z)_{0,1} \,:\, Z\in C^{1\var}([0,T];\mathbb{R}^d)\}\subset T^2(\mathbb{R}^d).
\end{equation*}
On $G^2(\mathbb{R}^d)$ we work with the Carnot--Caratheodory metric $d_{cc}$, which is given by
\begin{equation*}
  d_{cc}(g,h) := \|g^{-1} \otimes h\|_{cc}  \quad \text{for} \quad  g,h \in G^2(\R^d),
\end{equation*}
where $\| \cdot \|_{cc} $ is the Carnot--Caratheodory norm defined via \cite[Theorem~7.32]{Friz2010}, cf. \cite[Definition~7.41]{Friz2010}. The metric $d_{cc}$ turns $G^2(\R^d)$ into a complete geodesic metric space. For a path $\X\colon [0,T]\to G^2(\R^d)$, we set $\X_{s,t} := \X_s^{-1} \otimes \X_t$ for any subinterval $[s,t] \subset [0,T]$. We refer to \cite[Chapter~7]{Friz2010} for a more comprehensive introduction to $G^N(\R^d)$. 

\medskip

Analogously to \cite{Liu2018,Liu2021}, we want to consider rough paths with fractional Sobolev regularity. For this purpose, let us recall the definition of Sobolev regularity for functions mapping into a metric space~$(E,\textbf{d})$. For $\alpha \in (0,1)$, $p \in (1,+\infty)$ and a (continuous) function $f\colon [0,T]\to E$ we define the \textit{fractional Sobolev regularity} by
\begin{equation}\label{eq: fractional Sobolev norm}
 \|f\|_{W^{\alpha}_p}:= \|f\|_{W^{\alpha}_p;[0,T]}:= \bigg( \iint_{[0,T]^2} \frac{\textbf{d}(f(u),f(v))^p}{|v-u|^{\alpha p+1}} \dd u \dd v\bigg)^{1/p} + \bigg( \int_{[0,T]} \textbf{d}(x_0, f(u))^p \dd u \bigg)^{1/p}
\end{equation} 
and in the case of $p=+\infty$ we set 
\begin{equation*}
  \|f\|_{W^{\alpha}_p}:= \|f\|_{W^{\alpha}_p;[0,T]}:= \sup_{u,v\in[0,T],\,} \frac{\textbf{d}(f(u),f(v))}{|v-u|^\alpha} + \sup_{u \in [0,T]} \textbf{d}(x_0,f(u))
\end{equation*} 
for an arbitrary $x_0 \in E$. The latter case is also known as H{\"o}lder regularity. The space $W^{\alpha}_p([0,T];E)$ consists of all continuous functions $f\colon [0,T]\to E$ such that $\|f\|_{W^{\alpha}_p}<+\infty$. Notice that the fractional Sobolev space $W^{\alpha}_p([0,T];E)$ is independent of the reference point~$x_0$. The Sobolev regularity leads naturally to the notion of (fractional) Sobolev rough paths.

\begin{definition}[Sobolev rough path]\label{def:Sobolev rough path}
  Let $\alpha \in (1/3,1/2)$ and $p \in (1,+\infty]$ be such that $\alpha > 1/p$. The space $W^{\alpha}_p([0,T];G^{2}(\mathbb{R}^d))$ consists of all paths $\X\colon[0,T]\to G^{2}(\mathbb{R}^d) $ such that 
  \begin{equation*}
    \| \X \|_{W^{\alpha}_p} := \Big(\iint_{[0,T]^2}\frac{d_{cc}(\X_s,\X_t)^p}{|t-s|^{\alpha p + 1}} \dd s \dd t\Big)^{1/p} <+\infty.
  \end{equation*}
  The space $W^{\alpha}_p([0,T];G^{2}(\mathbb{R}^d))$ is called the \textit{weakly geometric Sobolev rough path space} and $\X\in W^{\alpha}_p([0,T];G^{2}(\mathbb{R}^d))$ is called a \textit{weakly geometric rough path of Sobolev regularity} $(\alpha,p)$ or in short \textit{Sobolev rough path}. 
\end{definition}

\begin{remark}
  Suppose that $\alpha \in (1/3,1/2)$ and $p \in (1,+\infty]$ be such that $\alpha > 1/p$, as done in Definition~\ref{def:Sobolev rough path}. By \cite[Theorem~2]{Friz2006}, every weakly geometric rough path of Sobolev regularity $(\alpha,p)$ is a continuous weakly geometric rough path of finite $1/\alpha$-variation. Therefore, the definition of the Sobolev regularity $\| \X \|_{W^{\alpha}_p}$ for a (continuous) rough path $\X$ does not require the additional $L^p$-regularity and, thus, we did drop it in definition of $\| \X \|_{W^{\alpha}_p}$, following the previous definitions of Sobolev rouh paths, see e.g. \cite[Definition~2.2]{Liu2021}.
\end{remark}

\subsection{Basic notation and function spaces}\label{subsec:basic notation}

As usual, $\mathbb{Z}$ denotes the set of integers, $\N :=\{1,2,\dots \}$ are the natural numbers and we set $\N_0:=\N \cup \{0\}$. For $z \in \R$ we write $\lfloor z \rfloor:= \max \{ y\in \mathbb{Z} \,:\, y\leq z\}$. The ball in $\mathbb{R}^{k}$, around $x\in\mathbb{R}^{k}$ with radius $R>0$ is denoted by $B(x,R)$. For two real functions $a,b$ depending on variables $x$ one writes $a\lesssim b$ or $a\lesssim_z b$ if there exists a constant $C(z)>0$ such that $a(x) \leq C(z)\cdot b(x)$ for all $x$, and $a\sim b$ if $a\lesssim b$ and $b\lesssim a$ hold simultaneously.

The space $L^{p}:=L^{p}(\mathbb{R}^{k},\d x)$, $p\geq1$, is the Lebesgue space, that is, the space of all functions~$f$ such that $\int_{\mathbb{R}^{k}}|f(x)|^{p}\dd x<+\infty$. We also set $L^q_\lambda := L^q((0,1),\lambda^{-1}\d \lambda)$ for $q\geq 1$ and write $L^{p}(\mathbb{R}^{k};B)$ for the $L^p$-space of functions $f\colon \R^k\to B$ where $B$ is a Banach space. The notation $\langle f, g\rangle$ is used for the $L^2$-inner product of $f$ and $g$ as well as the evaluation of the distribution~$f$ against the test function~$g$. 

The space $\ell^p$ is the Banach space of all sequences $(x_n)_{n\in \N}$ of real numbers such that $\sum_{n\in \N}|x_n|^p<+\infty$ and the corresponding norm is denoted by $\|\cdot \|_{\ell^p}$. The space $\ell^p_n$, for $n\in \N$, is the Banach space of all sequences $u(x) \in \R$, $x\in \Lambda_n :=  \{ 2^{-n }k\,:\, k \in \mathbb{Z} \}$, such that 
\begin{equation*}
  \|u(x)\|_{\ell^p_n} := \bigg(\sum_{x\in \Lambda_n }2^{-n d} |u(x)|^p\bigg)^{1/p}<+\infty.
\end{equation*}

The space $\mathcal{D}^{\prime}=\mathcal{D}^{\prime}(\mathbb{R}^{d})$ is the space of Schwartz distributions, that is, the topological dual of the space of compactly supported infinitely differentiable functions.

The space of H{\"o}lder continuous functions $\varphi\colon\mathbb{R}^{k}\to\mathbb{R}$ of order $r \geq 0$ is denoted by $\mathcal{C}^{r}$, that is, $\varphi$ is bounded (not necessarily continuous) if $r = 0$, H\"older continuous for $ 0 < r \leq 1$ (which amounts precisely to Lipschitz continuous for $r=1$, the derivative does not necessarily exist everywhere). For $ r > 1 $ not an integer the function~$\varphi$ is $\lfloor r\rfloor $-times continuously differentiable and the derivatives of order $\lfloor r\rfloor$ are H\"older continuous of order $r-\lfloor r\rfloor$. The space $\mathcal{C}^{r}$ is equipped with the norm
\begin{equation*}
  \lVert \varphi \rVert_{\mathcal{C}^r} := \sum_{j=0}^{\lfloor r \rfloor} \| D^j \varphi \|_\infty + \1_{r>\lfloor r \rfloor} \| D^{\lfloor r \rfloor} \varphi \|_{r - \lfloor r \rfloor},
\end{equation*}
where $\lVert \cdot \rVert_\beta$ denotes the $\beta$-H\"older norm for $\beta \in (0,1]$, and $\lVert \cdot \rVert_\infty$ denotes the supremum norm. If a function $\varphi\in\mathcal{C}^{r}$ has compact support, we say $\varphi\in\mathcal{C}_{0}^{r}$. Additionally, we use $\varphi\in\mathcal{B}^{r}$ if $\varphi\in\mathcal{C}_{0}^{r}$ is such that $\|\varphi\|_{\mathcal{C}^{r}}\leq 1$ and $\mathrm{supp}\,\varphi\subset B(0,1)$, and $\varphi\in\mathcal{B}^{r}_{n}$ for $n\in \N$ if $\varphi \in \mathcal{B}^{r}$ and $\varphi$ annihilates all polynomials of degree at most $n$. We set $\mathcal{B}^{r}_{-n}(\R^k) := \mathcal{B}^{r}(\R^k)$ for all $n \in \N$.

\section{Lifting Sobolev paths to Sobolev rough paths}\label{sec:lifting}

This section is devoted to show that every path of suitable Sobolev regularity can be lifted to a weakly geometric rough path \emph{possessing exactly the same Sobolev regularity}. To prove this statement, we proceed via an approach based on Hairer's reconstruction theorem appearing in the theory of regularity structures~\cite{Hairer2014}, which requires not only to use a Sobolev topology on the space of modelled distributions, as introduced in \cite{Hairer2017} and \cite{Liu2016} (see also \cite{Hensel2020}), but additionally to generalize the definition of models from the originally assumed H{\"o}lder bounds to more general Sobolev bounds. For a further discussion on this point we refer the end of Subsection~\ref{subsec:regularity structures}.

\subsection{Elements of regularity structures in a Sobolev stetting}\label{subsec:regularity structures}

In order to construct a Sobolev rough path lift of a Sobolev path relying on Hairer's theory of regularity structures, we introduce the essential ingredients of the theory in the following. For more detailed introductions we refer to \cite{Hairer2015,Chandra2017}. Let us start by recalling the definition of a regularity structure as given in \cite[Definition~2.1]{Hairer2014}.

\begin{definition}
  A triplet $\mathcal{T}=(A, T, G)$ is called \textit{regularity structure} if it consists of the following three objects:
  \begin{itemize}
    \item An \textit{index set} $A \subset\R$, which is locally finite\footnote{That is, $A$ does not contain any cluster point.} and bounded from below, with $0\in A$.
    \item A \textit{model space} $T=\bigoplus_{\alpha\in A}T_{\alpha}$, which is a graded vector space with each $T_{\alpha}$ a Banach space and $T_{0}\approx\R$. Its unit vector is denoted by $\mathbf{1}$.
    \item A \textit{structure group} $G$ consisting of linear operators acting on $T$ such that, for every $\Gamma\in G$, every $\alpha\in A$, and every $a\in T_{\alpha}$ it holds
          \begin{align*}
            \Gamma a-a\in\bigoplus_{\beta\in A;\,\beta<\alpha}T_{\beta}.
          \end{align*}
          Moreover, $\Gamma\mathbf{1}=\mathbf{1}$ for every $\Gamma\in G$. 
  \end{itemize}
  For any $\tau\in T$ and $\alpha\in A$ we denote by $\mathcal{Q}_{\alpha}\tau$ the projection of $\tau$ onto $T_{\alpha}$ and set $|\tau |_{\alpha}:=\|\mathcal{Q}_{\alpha}\tau\|$. Furthermore, for $\gamma > \min A$ we set $T_{\gamma}^{-}:=\bigoplus_{\alpha\in A_{\gamma}}T_{\alpha}$ where $A_{\gamma}:=\{\alpha\in A\,:\,\alpha<\gamma\}$.
\end{definition}

In view of the definition of (real-valued) Sobolev spaces (see e.g. \cite[Definition~2.1]{Hairer2017}) and of models with global bounds (see \cite[Definition~2.8]{Hairer2017}), we introduce a Sobolev version of models with global bounds.

\begin{definition}[Sobolev model]\label{def:Sobolev model}
  Let $\mathcal{T}=(A,T,G)$ be a regularity structure. For $p\in [1,+\infty]$ and $\gamma > \min A$ a \emph{Sobolev model} is a pair $(\Pi,\Gamma)$ that satisfies the following conditions:
  \begin{itemize}
    \item $\Pi = (\Pi_x)_{x\in \R^d}$ is a collection of linear maps $\Pi_x \colon \mathcal{T}_{\gamma}^{-}\to \mathcal{D}^\prime (\R^d)$ such that 
               \begin{equation*}
                 \|\Pi\|_{p} := \sup_{\zeta \in A_{\gamma}} \sup_{\tau \in T_{\zeta}} 
                 |\tau |_{\zeta}^{-1}\bigg \| \bigg \| \sup_{\eta \in \mathcal{B}^{r}_{[\zeta]}(\R^d)}  
                 \frac{|\langle \Pi_x \tau, \eta^{\lambda}_x \rangle |}{\lambda^{\zeta} } \bigg \|_{L^p(\d x)} \bigg \|_{L^p_{\lambda}} 
                 <+\infty,
               \end{equation*}
               where $\eta^\lambda_x$ is defined below in Definition 3.6 and $r = \lfloor |\min A| \rfloor + 1$.
    \item $\Gamma = (\Gamma_{x,y})_{x,y\in \R^d}$ fulfills $\Gamma_{x,y}\in G$ for all $x,y\in \R^d$ and 
                \begin{equation*}
                  \|\Gamma \|_{p} := \sup_{\beta <\zeta \in A_{\gamma}}\sup_{\tau \in T_{\zeta}}
                  |\tau |_{\zeta}^{-1}\bigg \| \bigg \|
                  \frac{|\Gamma_{x,x+h} \tau|_{\beta}}{\|h\|^{\zeta-\beta}} \bigg \|_{L^p(\d x)} \bigg \|_{L^p_{h}}
                  <+\infty,
                \end{equation*}
                where $\|g\|_{L^p_h} = (\int_{B(0,1)} |g(h)|^p \frac{\d h}{\|h\|})^{1/p}$.
  \end{itemize} 
\end{definition}

\begin{remark}
  The Sobolev model could also be defined locally in the sense that the $L^p$-norm with respect to~$x$ is taken on compact subsets of~$\R^d$, which is closer to the original definition of models given in \cite[Definition~2.17]{Hairer2014}. However, for our purpose the global bounds are the more convenient ones. Moreover, a non-Euclidean scaling can be included in Definition~\ref{def:Sobolev model} and the extension to more general Besov bounds can be achieved by replacing the $L^p_{\lambda}$-norm by an $L^q_{\lambda}$-norm for $q\in [1,+\infty]$.
\end{remark}

\begin{remark}\label{rmk:discussion Sobolev model}
  While the definition of Sobolev models seems to be the canonical one for our later choice of a regularity structure, cf. Example~\ref{ex:basic Sobolev model} below, in general different regularity structures might lead to other natural choices of models with Sobolev type bounds. 
\end{remark}

Following~\cite{Hairer2017} and \cite{Liu2016}, we introduce the Sobolev space of modelled distributions. Notice that the definition of modelled distributions depends on the definition of models and thus the generalized definition of models in Definition~\ref{def:Sobolev model} also leads to more general notion of modelled distributions. 

\begin{definition}\label{def:Sobolev modelled distributions}
  Let $\mathcal{T}=(A,T,G)$ be a regularity structure with a model $(\Pi,\Gamma)$, $\gamma\in \R$ and $p,q\in [1,+\infty)$. The \textit{Besov space}~$\mathcal{D}_{p,q}^{\gamma}$ consists of all measurable functions $f\colon\R^{d}\to T_{\gamma}^{-}$ such that 
  \begin{align*} 
  \begin{split}
    \interleave f \interleave_{\gamma,p,q} 
    := & \sum_{\alpha \in A_\gamma}\| |f(x)|_{\alpha} \|_{L^p(\d x)}\\
    & + \sum_{\alpha \in A_\gamma}\Big(\int_{h \in B(0,1)} \left\|\frac{|f(x+h) - \Gamma_{x+h,x}f(x)|_\alpha}{\|h\|^{\gamma - \alpha}}\right\|_{L^p(\d x)}^q\frac{\d h}{\|h\|^d}\Big)^{\frac{1}{q}}<+\infty .
  \end{split}
  \end{align*}
  We refer to $\mathcal{D}_{p,p}^{\gamma}$ as \textit{Sobolev space of modelled distributions}.
\end{definition}

The corresponding Sobolev and Besov spaces consisting of real-valued distributions are introduced in the next definition following \cite[Definition~2.1]{Hairer2017}. For a more comprehensive treatment of these function spaces we refer to \cite{Triebel2010}.

\begin{definition}\label{def: Real valued Besov spaces}
  Let $\alpha <0 $, $p,q\in [1,+\infty)$ and $r\in \N$ such that $r>|\alpha|$, and define
  \begin{equation*}
    \eta^\lambda_x(y):= \lambda^{-k}\eta (\lambda^{-1}(y_1-x_1),\dots,\lambda^{-1}(y_k-x_k))
  \end{equation*}
  for $\lambda\in (0,1]$, $x=(x_1,\dots,x_k)\in \R^k$ and $y=(y_1,\dots,y_k)\in \R^k$.

  For $\alpha <0 $, the \textit{Besov space}~$\mathcal{B}^{\alpha}_{p,q}:=\mathcal{B}^{\alpha}_{p,q}(\R^k)$ is the space of all distributions~$\xi$ on $\R^k$ such that
  \begin{equation*}
    \bigg\| \Big\| \sup_{\eta\in \mathcal{B}^r(\R^k)} \frac{|\langle \xi , \eta^\lambda_x \rangle|}{\lambda^\alpha}  \Big\|_{L^p(\d x)} \bigg\|_{L^q_\lambda} <+\infty.
  \end{equation*}
  For $\alpha \geq 0 $, the \textit{Besov space}~$\mathcal{B}^{\alpha}_{p,q}:=\mathcal{B}^{\alpha}_{p,q}(\R^k)$ is the space of all distributions~$\xi$ on $\R^k$ such that
  \begin{equation*}
    \bigg\| \sup_{\eta\in \mathcal{B}^r(\R^k)}|\langle \xi , \eta^\lambda_x \rangle|  \Big\|_{L^p(\d x)}  < +\infty
    \quad \text{and}\quad
    \bigg\| \Big\| \sup_{\eta\in \mathcal{B}^r_{\lfloor \alpha \rfloor}(\R^k)} \frac{|\langle \xi , \eta^\lambda_x \rangle|}{\lambda^\alpha}  \Big\|_{L^p(\d x)} \bigg\|_{L^q_\lambda} < +\infty.
  \end{equation*}
  The \textit{Sobolev space} $W^{\alpha}_p$ is defined as $W^{\alpha}_p:=W^{\alpha}_p(\R^k):= \mathcal{B}^{\alpha}_{p,p}(\R^k)$.
\end{definition}

A function $f\colon \mathbb{R} \to \mathbb{R}$ has the fractional Sobolev regularity defined in \eqref{eq: fractional Sobolev norm} (with $[0,T]$ replaced by $\mathbb{R}$ and $\textbf{d} = |\cdot|$ induced by the Euclidean norm on $\R$) if and only if it is an element in $W^\alpha_p(\mathbb{R})$ in the sense of Definition \ref{def: Real valued Besov spaces} and the two norms are equivalent, see e.g. \cite{Triebel2010} and \cite{Simon1990}. Note that these two norms remain equivalent when they are restricted to bounded interval $[0,T]$, see e.g. \cite[Theorem 1.11]{Schneider2011}. In fact, thanks to \cite[Theorem 1.9]{Schneider2011} for any fractional Sobolev function $f$ defined on $[0,T]$ satisfying \eqref{eq: fractional Sobolev norm}, there exists a bounded extension operator $\text{Ext}$ such that $g = \text{Ext}(f)$ is a fractional Sobolev function defined on $\R$ and $g|_{[0,1]} = f$. This allows us to  work globally with $g$ instead of $f$. In view of this observation, subsequently we will not distinguish the ``localized function''~$f$ and its extension~$g$ defined on~$\R$.

The construction of a rough path lift for a $\R^d$-valued path with suitable Sobolev regularity is based on the following regularity structure.

\begin{example}\label{ex:basic Sobolev model}
  Let $\alpha \in (1/3,1/2)$ and $p \in (1,+\infty)$ such that $\alpha > 1/p$ and suppose that $W \in W^{\alpha}_p(\R)$. The path~$W$ induces a regularity structure $(A, T, G)$ via 
  \begin{equation*}
    A = \{\alpha - 1, 0\},\quad
    T = T_{\alpha - 1} \oplus  T_{0} 
    = \langle \dot{\mathbb{W}} \rangle \oplus \langle \1 \rangle,\quad
    G = \{\text{Id}_{T}\}
  \end{equation*}
  and an associated Sobolev model $(\Pi, \Gamma)$ via 
  \begin{equation*}
     \Pi_x(\dot{\mathbb{W}}) := \dot{W},\quad \Pi_x(\1) := 1 \in \R
     \quad \text{and}\quad  \Gamma_{x,y} := \text{Id}_{T}, \quad  \text{for all } x,y \in \R, 
  \end{equation*}
  where $\dot{W}$ stands for the distributional derivative of~$W$. Indeed, for $\tau = \tau_0\1\in T_0$ with $\tau_0 \in \R$ we have 
  \begin{equation*}
    \bigg \| \bigg \| \sup_{\eta \in \mathcal{B}^{r}_{0}(\R^d)}  \frac{|\langle \Pi_x \tau, \eta^{\lambda}_x \rangle |}{\lambda^{0} } \bigg \|_{L^p(\d x)} \bigg \|_{L^p_{\lambda}} = 0,
  \end{equation*}
  since in this case any test function $\eta \in \mathcal{B}^{r}_{0}(\R)$ annihilating constants has a vanishing mean. For $\tau = \tau_{\alpha - 1}\dot{\mathbb{W}} \in \Tc_{\alpha - 1}$ with $\tau_{\alpha - 1} \in \R$ we have
  \begin{align*}
    \bigg \| \bigg \| \sup_{\eta \in \mathcal{B}^{r}(\R^d)}  
    \frac{|\langle \Pi_x \tau, \eta^{\lambda}_x \rangle |}{\lambda^{\alpha - 1} } \bigg \|_{L^p(\d x)} \bigg \|_{L^p_{\lambda}} &= |\tau_{\alpha - 1}|\bigg \| \bigg \| \sup_{\eta \in \mathcal{B}^{r}(\R^d)}
    \frac{|\langle \dot{W}, \eta^{\lambda}_x \rangle |}{\lambda^{\alpha - 1} } \bigg \|_{L^p(\d x)} \bigg \|_{L^p_{\lambda}} \\
    &=|\tau_{\alpha - 1}|\|\dot{W}\|_{W^{\alpha-1}_p} \lesssim |\tau_{\alpha - 1}|,
  \end{align*}
  since $\dot{W} \in W^{\alpha - 1}_p(\R)$ for $W \in W^{\alpha}_p(\R)$ by \cite[Theorem~2.3.8]{Triebel2010}. This shows that $\|\Pi\|_{p} = \|\dot{W}\|_{W^{\alpha-1}_p}$. The estimate for $\|\Gamma\|_{p}$ holds since $|\Gamma_{x,y}\tau|_{\beta} = |\tau|_{\beta} = 0$ for any $\tau \in \Tc_{\zeta}$, $\beta < \zeta$ and $x,y \in \R$.
\end{example}

Given a two-dimensional path $(Y,W)\in W^{\alpha}_p$, in order to construct a rough path lift via Hairer's theory of regularity structure, the key idea goes as follows, cf. \cite[Proposition~13.23]{Friz2014}: $W$ induces a regularity $(\mathcal{A},\Tc,\mathcal{G})$ and a model $(\Pi,\Gamma)$ as defined in Example~\ref{ex:basic Sobolev model} and $Y$ induces a modelled distribution with negative regularity in the sense of Definition~\ref{def:Sobolev modelled distributions}. Then, in the case of a H{\"o}lder continuous path $(Y,W)$ an application of Hairer's reconstruction operator (\cite[Theorem~3.10]{Hairer2014}) leads to a rough path lift with the same H{\"o}lder regularity. 
  
While the reconstruction theorem for modelled distributions with negative Sobolev regularity (but for the original H\"older type models) was recently established in \cite[Theorem~2.11]{Liu2016}, it is not sufficient to lift a Sobolev path to a rough path with the same Sobolev regularity. First, one loses already regularity when constructing a H{\"o}lder type model starting with a Sobolev path. Second, the classical bounds (relying on H{\"o}lder type models) obtained for the reconstruction operator, see \cite[Theorem~2.11]{Liu2016}, are not sufficient and would lead again to a loss of regularity. 

As a consequence, we have to derive sharper bounds for the reconstruction operator for lifting Sobolev paths to rough path in the case of Sobolev models, see \eqref{eq:Bound of reconstruction operator for Besov modelled distribution} and Remark~\ref{rmk:new bounds} below.

\subsection{Sobolev rough path lift via the reconstruction operator}\label{subsec:lift via regularity structures}

In this subsection we construct a Sobolev rough path lift of a $\R^d$-valued Sobolev path with regularity $\alpha \in (1/3,1/2)$.

\begin{theorem}\label{thm:Sobolev path lift}
  Let $\alpha \in (1/3, 1/2)$ and $p \in (1,+\infty]$ be such that $\alpha > 1/p$. For every Sobolev path $X\in  W^{\alpha}_p([0,T];\R^d)$ there exists a rough path lift $\X := (X,\mathbb{X}) \in W^{\alpha}_p([0,T];G^{2}(\R^d))$ of $X$.
\end{theorem}

Let us first observe that it is sufficient to prove Theorem~\ref{thm:Sobolev path lift} for a $\R^2$-valued path $X = (Y,W)\in W^{\alpha}_p([0,T];\R^2)$. The $d$-dimensional case immediately follows from successively applying the $2$-dimensional case. Secondly, we extend $X$ continuously from $[0,T]$ to $\R$ such that $X=(Y,W) \in W^{\alpha}_p(\R;\R^2)$ for $\alpha \in (1/3, 1/2)$ and $\alpha > 1/p$. By classical Besov embeddings (see e.g. \cite{Triebel2010}), we note that $X \in \mathcal{B}^{\alpha - 1/p}_{\infty, \infty}$ and thus $\sup_{x \in \R} |X(x)| <+ \infty$. 

\smallskip

Let $(\mathcal{A},\Tc,\mathcal{G})$ be the regularity structure induced by the second component~$W$ with the corresponding model $(\Pi,\Gamma)$ as defined in Example~\ref{ex:basic Sobolev model}. The first component~$Y$ induces a modelled distribution $\dot{\mathbb{Z}}\colon\R \rightarrow \Tc$ by setting $\dot{\mathbb{Z}}(x) := Y_x\dot{\mathbb{W}}$ for $x\in \R$. Then, we have $\dot{\mathbb{Z}} \in \mathcal{D}^{\gamma}_{p,p}$ with $\gamma := 2\alpha - 1$ in the sense of Definition~\ref{def:Sobolev modelled distributions}. Indeed, note that with $\zeta = \alpha -1$ the translation bound of $\dot{\mathbb{Z}}$ is equal to
\begin{align*}
  \Big(\int_{h \in [-1,1]} \Big\|&\frac{| \dot{\mathbb{Z}}(x + h) - \Gamma_{x+h,x}\dot{\mathbb{Z}}(x)|_{\zeta}}{|h|^{\gamma - \zeta}}\Big\|^p_{L^p(\d x)}\,\frac{\d h}{|h|}\Big)^{\frac{1}{p}} \\
  &= \Big(\int_{h \in [-1,1]} \Big\|\frac{|Y(x + h) - Y(x)|}{|h|^{\alpha}}\Big\|^p_{L^p(\d x)}\,\frac{\d h}{|h|}\Big)^{\frac{1}{p}} 
  \lesssim \|Y\|_{W^{\alpha}_p} <+ \infty,
\end{align*}
where the inequality follows from the equivalence of Sobolev norms, see e.g. \cite{Triebel2010} and \cite{Simon1990}. Similarly one can easily show that $\dot{\mathbb{Z}}$ also has a bounded $L^p$--norm as $Y \in W^\alpha_p$ satisfies this property by definition.

Before coming to the actual proof of Theorem~\ref{thm:Sobolev path lift}, we need to establish the analog of the reconstruction theorem similar to \cite[Theorem~2.11]{Liu2016}, that is, we need to show the existence of the reconstruction operator~$\mathcal{R}$, which is required for lifting Sobolev paths to Sobolev rough paths (Lemma~\ref{lem:regularity of Rf}), mapping modelled distributions into a Sobolev space and the required bound~\eqref{eq:Bound of reconstruction operator for Besov modelled distribution} below (Lemma~\ref{lem:Bound in reconstruction theorem w.r.t. Sobolev model}).

Namely, for the regularity structure $(\mathcal{A},\Tc,\mathcal{G})$ and the Sobolev model $(\Pi,\Gamma)$ as defined Example~\ref{ex:basic Sobolev model}, there exists a distribution $\mathcal{R}\dot{\mathbb{Z}}\in W^{\alpha - 1}_p$ satisfying that
\begin{equation}\label{eq:Bound of reconstruction operator for Besov modelled distribution}
  \left\lVert \Big\| \sup_{\eta \in \mathcal{B}^r} \frac{|\langle \mathcal{R}\dot{\mathbb{Z}} - \Pi_x\dot{\mathbb{Z}}(x), \eta^\lambda_x \rangle|}{\lambda^\gamma} \Big\|_{L^{\frac{p}{2}}(\d x)} \right\rVert_{L^{\frac{p}{2}}_\lambda} \lesssim \|\Pi\|_{p}(1 + \|\Gamma\|_{p})\vertiii{\dot{\mathbb{Z}}}_{\gamma,p,p}.
\end{equation}
Note that $L^{p/2}$-norms are used in the Estimate~\eqref{eq:Bound of reconstruction operator for Besov modelled distribution} instead of $L^{p}$-norms, as usually obtained for the reconstruction operator, see Remark~\ref{rmk:new bounds} below.

In order to define $\mathcal{R}\dot{\mathbb{Z}}$, let $r \in \N$ be such that $r > |\alpha - 1- \frac{1}{p}|$ (we will see later why such special $r$ is needed). We fix $\varphi\colon\mathbb{R}\to\mathbb{R}$ and $\psi\colon\mathbb{R}\to\mathbb{R}$ both in $\mathcal{C}^r_0$ as the father wavelet and mother wavelet, respectively, of a wavelet analysis on $\R$ which has the following properties:
\begin{enumerate}
  \item For every polynomial $P$ of degree at most~$r$ there exists a polynomial $\hat{P} $ such that 
        \begin{equation*}
          \sum_{y\in\mathbb{Z}}\hat{P}(y) \,\varphi(x-y) =P(x), \quad x\in\mathbb{R}.
        \end{equation*}
  \item For every $y\in\mathbb{Z}$ one has $\int_{\mathbb{R}}\varphi(x)\varphi(x-y)\dd x=\delta_{y,0}$.
  \item There exist coefficients $(a_{k})_{k\in\mathbb{Z}}$ with only finitely many non-zero values such that 
        \begin{equation*}
          \varphi(x)=\sum_{k\in\mathbb{Z}}a_{k}\varphi(2x-k),\quad x\in \R. 
        \end{equation*}
  \item The function~$\psi$ annihilates all polynomials of degree at most~$r$.
  \item For any $n\geq 0$, the set 
        \begin{equation*}
          \{\varphi_{x}^{n}\,:\, x\in\Lambda_{n}\}\cup\{ \psi_{x}^{m}\,:\, x\in\Lambda_{m},\, m\geq n\}
        \end{equation*}
        constitutes an orthonormal basis of~$L^{2}$.
\end{enumerate}
Here we used the notation
\begin{equation*}
  \varphi_{x}^{n}(y):= 2^{\frac{n}{2}} \varphi(2^{n}(y-x))\quad \text{and}\quad  \psi_{x}^{n}(y):= 2^{\frac{n}{2}} \psi(2^{n}(y-x)),
\end{equation*}
for $x,y\in\R$ and $\Lambda_n := \{ 2^{-n }k\,:\, k \in \mathbb{Z} \}$. For more details on wavelet analysis we refer the reader to \cite{Meyer1992} and \cite{Daubechies1988} or in our particular setting to \cite[Section~2.1]{Hairer2017}.

As in the proof of \cite[Theorem~2.11]{Liu2016}, we define
\begin{equation}\label{eq:reconstrcution operator for Besov model}
  \mathcal{R} f := \sum_{n\in \N_0} \sum_{x\in \Lambda_n}  \langle \Pi_x \overline{f}^n (x), \psi^n_x \rangle   \psi^n_x + \sum_{x \in \Lambda_0}\langle \Pi_x \overline{f}^0 (x), \varphi^0_x   \rangle   \varphi^0_x,
\end{equation}
where $f := \dot{\mathbb{Z}}$ and $\bar{f}^n(x) := \int_{B(x,2^{-n})}2^n\Gamma_{x,y}f(y) \dd y$ for $x \in \Lambda_n$, cf. \cite[(2.8)]{Hairer2017}.
 
\begin{lemma}\label{lem:regularity of Rf}
  The distribution $\mathcal{R}f$ defined in~\eqref{eq:reconstrcution operator for Besov model} is well-defined and belongs to $W^{\alpha - 1}_p$.
\end{lemma}

\begin{proof}
  We set for every $n \geq 0$, $x \in \Lambda_n$ a real number
  \begin{equation*}
    a^{n,\psi}_x := \langle \mathcal{R}f, \psi^n_x \rangle = \langle \Pi_x\overline{f}^n(x), \psi^n_x \rangle,
  \end{equation*}
  and for $x \in \Lambda_0$, $b^0_x := \langle \mathcal{R}f, \varphi_x \rangle = \langle \Pi_x\overline{f}^0(x), \varphi_x \rangle$. Invoking \cite[(2.2)]{Hairer2017}, it suffices to show that 
  \begin{equation*}
    \left\lVert \Big\|\frac{a^{n,\psi}_x}{2^{-\frac{n}{2}- n(\alpha - 1)}}\Big\|_{\ell^p_n} \right\rVert_{\ell^p} <+ \infty
    \quad\text{and}\quad
    \Big\| b^0_x \Big\|_{\ell^p_0} < +\infty.
  \end{equation*}
  To this end, we remark that by the definition of $\overline{f}^n$ and the fact that in our setting $\Gamma_{x,y} = \text{Id}_{\Tc}, \Pi_xf(y) = Y_y\dot{W}$, it holds that
  \begin{align*}
    |a^{n,\psi}_x|  
    \leq \int_{B(x,2^{-n})} 2^{n}|\langle \Pi_x\Gamma_{x,y}f(y), \psi^n_x \rangle| \dd y
    = \int_{B(x,2^{-n})} 2^{n}|Y_y||\langle \dot{W}, \psi^n_x \rangle| \dd y.
  \end{align*}
  It follows that
  \begin{align*}
    \Big\|\frac{a^{n,\psi}_x}{2^{-\frac{n}{2}- n(\alpha - 1)}}\Big\|_{\ell^p_n} 
    &\le \Big(\sum_{x \in \Lambda_n}2^{-n}\Big(\int_{B(x,2^{-n})} 2^{n}|Y_y|\frac{|\langle \dot{W}, \psi^n_x \rangle|}{2^{-\frac{n}{2}- n(\alpha - 1)}} \dd y\Big)^p\Big)^{\frac{1}{p}} \\
    &\lesssim |Y|_{\infty}\Big(\sum_{x \in \Lambda_n}2^{-n}\Big(\frac{|\langle \dot{W}, \psi^n_x \rangle|}{2^{-\frac{n}{2}- n(\alpha - 1)}}\Big)^p\Big)^{\frac{1}{p}}
  \end{align*}
  and therefore
  \begin{equation*}
    \left\lVert \Big\|\frac{a^{n,\psi}_x}{2^{-\frac{n}{2}- n(\alpha - 1)}}\Big\|_{\ell^p_n} \right\rVert_{\ell^p} \lesssim \left\lVert \Big\|\frac{|\langle \dot{W}, \psi^n_x \rangle|}{2^{-\frac{n}{2}- n(\alpha - 1)}}\Big\|_{\ell^p_n} \right\rVert_{\ell^p} <+ \infty,
  \end{equation*}
  since $W \in W^{\alpha}_p$ by assumption. The same argument gives us $\| b^0_x \|_{\ell^p_0} < +\infty$. Hence, we conclude that $\mathcal{R}f \in W^{\alpha - 1}_p$ by using \cite[Proposition~2.4]{Hairer2017}.
\end{proof}

As a next step we show Bound~\eqref{eq:Bound of reconstruction operator for Besov modelled distribution} for our Sobolev model.

\begin{lemma}\label{lem:Bound in reconstruction theorem w.r.t. Sobolev model}
  The distribution $\mathcal{R}f$ defined in \eqref{eq:reconstrcution operator for Besov model} satisfies Bound~\eqref{eq:Bound of reconstruction operator for Besov modelled distribution}.
\end{lemma}

\begin{proof}
  For fixed $x \in \R$, $\lambda \in (0,1]$ and $\eta \in \mathcal{B}^r$, we have
  \begin{equation*}  
    \langle \mathcal{R}f - \Pi_xf(x), \eta^\lambda_x\rangle = \sum_{n \geq 0} \sum_{y \in \Lambda_n}\langle \mathcal{R}f - \Pi_xf(x), \psi^n_y\rangle\langle\psi^n_y,\eta^\lambda_x\rangle + \sum_{y \in \Lambda_0}\langle \mathcal{R}f - \Pi_xf(x), \varphi_y\rangle \langle \varphi_y, \eta^\lambda_x\rangle, 
  \end{equation*}
  where in our case 
  \begin{align*}
    \langle \mathcal{R}f - \Pi_xf(x), \psi^n_y\rangle &= \langle \Pi_y\overline{f}^n(y) - \Pi_xf(x), \psi^n_y \rangle \\
    &= \int_{B(y,2^{-n})} 2^{n}\langle \Pi_y(\Gamma_{y,z}f(z) - \Gamma_{y,x}f(x)), \psi^n_y \rangle \dd z \\
    &= \int_{B(y,2^{-n})} 2^{n}\langle (Y_z - Y_x)\dot{W}, \psi^n_y \rangle \dd z
  \end{align*}
  and the same expression holds for $\langle \mathcal{R}f - \Pi_xf(x), \varphi_y\rangle$. It follows that
  \begin{equation}\label{eq:Bound for the inner product against single mother wavalet}
    |\langle \mathcal{R}f - \Pi_xf(x), \psi^n_y\rangle| \le \int_{B(y,2^{-n})} 2^{n}|Y_z - Y_x||\langle \dot{W}, \psi^n_y \rangle| \dd z.
  \end{equation}
  As in the proof of \cite[Theorem~3.1]{Liu2016} we use $\| \cdot\|_{L^q_{n_0}(\d\lambda)}$ to denote the $L^q$-norm with respect to the finite measure (with the total mass $\ln2$) $\lambda^{-1}\1_{(2^{-n_0-1},2^{-n_0}]}\dd\lambda$, and we consider two quantities 
  \begin{equation}\label{eq:the sum up to n0}
    \left\lVert \Big\| \sup_{\eta \in \mathcal{B}^r}\frac{|\sum_{n \le n_0}\sum_{y \in \Lambda_n}\langle \mathcal{R}f - \Pi_xf(x), \psi^n_y \rangle \langle \psi^n_y, \eta^\lambda_x \rangle |}{\lambda^\gamma} \Big\|_{L^{\frac{p}{2}}(\d x)} \right\rVert_{L^{\frac{p}{2}}_{n_0}(\d \lambda)}
  \end{equation}
  and 
  \begin{equation}\label{eq:the sum up after n0}
    \left\lVert \Big\| \sup_{\eta \in \mathcal{B}^r}\frac{|\sum_{n > n_0}\sum_{y \in \Lambda_n}\langle \mathcal{R}f - \Pi_xf(x), \psi^n_y \rangle \langle \psi^n_y, \eta^\lambda_x \rangle |}{\lambda^\gamma} \Big\|_{L^{\frac{p}{2}}(\d x)} \right\rVert_{L^{\frac{p}{2}}_{n_0}(\d \lambda)}.
  \end{equation}
  Since $\left\lVert \Big\| \sup_{\eta \in \mathcal{B}^r} \frac{|\langle \mathcal{R}f - \Pi_xf(x), \eta^\lambda_x \rangle|}{\lambda^\gamma} \Big\|_{L^{\frac{p}{2}}(\d x)} \right\rVert_{L^{\frac{p}{2}}_\lambda}$ is bounded by the sum of the $\ell^{\frac{p}{2}}(n_0 \in \N_0)$-norms of~\eqref{eq:the sum up to n0} and~\eqref{eq:the sum up after n0}, it suffices to establish the bound~\eqref{eq:Bound of reconstruction operator for Besov modelled distribution} for the $\ell^{\frac{p}{2}}(n_0 \in \N_0)$-norm of each term. (One can easily bound the terms in the expansion of $\mathcal{R}f - \Pi_xf(x)$ involved with $\varphi^0_y$ by using a similar argument.)  \smallskip
  
  \textit{Step~1:} We first give an estimate for the Term~\eqref{eq:the sum up to n0}. As in the proof of \cite[Theorem 3.1]{Hairer2017}, we note that for $\lambda \in (2^{-n_0-1},2^{-n_0}]$ and $n \le n_0$ one has $|\langle \psi^n_y , \eta^{\lambda}_x \rangle| \lesssim 2^{n/2}$ uniformly over all $y \in \Lambda_n$, $\eta \in \mathcal{B}^r$, $x \in \R$ and $n \le n_0$. Moreover, this inner product vanishes as soon as $|x - y| > C2^{-n}$ for some constant~$C$. Hence, inserting Inequality~\eqref{eq:Bound for the inner product against single mother wavalet} we obtain that
  \begin{align*}
    &\left\lVert \Big\| \sup_{\eta \in \mathcal{B}^r}\frac{|\sum_{n \le n_0}\sum_{y \in \Lambda_n}\langle \mathcal{R}f - \Pi_xf(x), \psi^n_y \rangle \langle \psi^n_y, \eta^\lambda_x \rangle |}{\lambda^\gamma} \Big\|_{L^{\frac{p}{2}}(\d x)} \right\rVert_{L^{\frac{p}{2}}_{n_0}(\d \lambda)} \\
    &\qquad\lesssim \sum_{n \le n_0} \Big\|\sum_{y \in \Lambda_n, |y - x| \le C2^{-n}}\int_{B(y,2^{-n})} 2^{n}\frac{|Y_z - Y_x||\langle \dot{W}, \psi^n_y \rangle|}{2^{-n_0\gamma -\frac{n}{2}}} \dd z \Big\|_{L^{\frac{p}{2}}(\d x)} \\
    &\qquad\lesssim \sum_{n \le n_0} 2^{(n_0-n)\gamma}\Big\|\int_{B(x,C^\prime 2^{-n})} 2^{n}\frac{|Y_z - Y_x|\sum_{y \in \Lambda_n, |y - x| \le C2^{-n}}|\langle \dot{W}, \psi^n_y \rangle|}{2^{-n\gamma -\frac{n}{2}}} \dd z \Big\|_{L^{\frac{p}{2}}(\d x)} \\
    &\qquad\lesssim \sum_{n \le n_0} 2^{(n_0-n)\gamma}\Big\|\int_{B(0,C^\prime 2^{-n})} 2^{n}\frac{|Y_{x+h} - Y_x|}{2^{-n\alpha}} \dd h\sum_{y \in \Lambda_n, |y - x| \le C2^{-n}}\frac{|\langle \dot{W}, \psi^n_y \rangle|}{2^{-n(\alpha - 1) -\frac{n}{2}}} \Big\|_{L^{\frac{p}{2}}(\d x)},
  \end{align*}
  where we used $\gamma = 2\alpha - 1 = \alpha + (\alpha - 1)$ in the last line.
  
  For each $n \le n_0$, by the above observation we can further deduce that
  \begin{align*}
    &\Big\|\int_{B(0,C^\prime 2^{-n})} 2^{n}\frac{|Y_{x+h} - Y_x|}{2^{-n\alpha}} \dd h\sum_{y \in \Lambda_n, |y - x| \le C2^{-n}}\frac{|\langle \dot{W}, \psi^n_y \rangle|}{2^{-n(\alpha - 1) -\frac{n}{2}}} \Big\|_{L^{\frac{p}{2}}(\d x)} \\
    &\quad\lesssim \int_{B(0,C^\prime 2^{-n})}2^{n} \Big(\int_{\R}\Big(\frac{|Y_{x+h} - Y_x|}{2^{-n\alpha}}\Big)^{\frac{p}{2}}\Big(\sum_{y \in \Lambda_n, |y - x| \le C2^{-n}}\frac{|\langle \dot{W}, \psi^n_y \rangle|}{2^{-n(\alpha - 1) -\frac{n}{2}}}\Big)^{\frac{p}{2}} \dd x\Big)^{\frac{2}{p}} \dd h \\
    &\quad \lesssim \int_{B(0,C^\prime 2^{-n})}2^{n} \Big(\int_{\R}\Big(\frac{|Y_{x+h} - Y_x|}{2^{-n\alpha}}\Big)^{p} \dd x\Big)^{\frac{1}{p}} \Big(\int_{\R}\Big(\sum_{y \in \Lambda_n, |y - x| \le C2^{-n}}\frac{|\langle \dot{W}, \psi^n_y \rangle|}{2^{-n(\alpha - 1) -\frac{n}{2}}}\Big)^{p} \dd x\Big)^{\frac{1}{p}} \dd h,
  \end{align*}
  where we used Minkowski's integral inequality for the measures $\dd x$ and $2^n\dd h$ on ${B(0,C^\prime2^{-n})}$, and the H{\"o}lder inequality in the last inequality.
  
  Now we look at the term $\int_{\R}\Big(\sum_{y \in \Lambda_n, |y - x| \le C2^{-n}}\frac{|\langle \dot{W}, \psi^n_y \rangle|}{2^{-n(\alpha - 1) -\frac{n}{2}}}\Big)^{p} \dd x$. It can be written as
  \begin{equation*}
    \sum_{z \in \Lambda_n} \int_{x \in B(z,2^{-n-1})} \Big(\sum_{y \in \Lambda_n, |y - x| \le C2^{-n}}\frac{|\langle \dot{W}, \psi^n_y \rangle|}{2^{-n(\alpha - 1) -\frac{n}{2}}}\Big)^{p} \dd x,
  \end{equation*}
  which can be bounded by $\sum_{z \in \Lambda_n} \int_{x \in B(z,2^{-n-1})} \Big(\sum_{y \in \Lambda_n, |y - z| \le C^\prime 2^{-n}}\frac{|\langle \dot{W}, \psi^n_y \rangle|}{2^{-n(\alpha - 1) -\frac{n}{2}}}\Big)^{p} \dd x$ for some suitable constant $C^\prime$ independent of $z \in \Lambda_n$. Therefore, we have
  \begin{equation*}
    \int_{\R}\Big(\sum_{y \in \Lambda_n, |y - x| \le C2^{-n}}\frac{|\langle \dot{W}, \psi^n_y \rangle|}{2^{-n(\alpha - 1) -\frac{n}{2}}}\Big)^{p} \dd x \le \sum_{z \in \Lambda_n} 2^{-n}\Big(\sum_{y \in \Lambda_n, |y - z| \le C^\prime 2^{-n}}\frac{|\langle \dot{W}, \psi^n_y \rangle|}{2^{-n(\alpha - 1) -\frac{n}{2}}}\Big)^{p}.
  \end{equation*}
  Since the cardinality of $\{y \in \Lambda_n, |y - z| \le C^\prime 2^{-n}\}$ is controlled by $C^\prime$, it yields that 
  \begin{equation*}
    \Big(\sum_{y \in \Lambda_n, |y - z| \le C^\prime 2^{-n}}\frac{|\langle \dot{W}, \psi^n_y \rangle|}{2^{-n(\alpha - 1) -\frac{n}{2}}}\Big)^{p} \lesssim \sum_{y \in \Lambda_n, |y - z| \le C^\prime 2^{-n}}\Big(\frac{|\langle \dot{W}, \psi^n_y \rangle|}{2^{-n(\alpha - 1) -\frac{n}{2}}}\Big)^{p}
  \end{equation*}
  and then a basic combinatorial argument gives that
  \begin{equation*}
    \sum_{z \in \Lambda_n} 2^{-n}\sum_{y \in \Lambda_n, |y - z| \le C^\prime 2^{-n}}\Big(\frac{|\langle \dot{W}, \psi^n_y \rangle|}{2^{-n(\alpha - 1) -\frac{n}{2}}}\Big)^{p} \lesssim \sum_{y \in \Lambda_n} 2^{-n}\Big(\frac{|\langle \dot{W}, \psi^n_y \rangle|}{2^{-n(\alpha - 1) -\frac{n}{2}}}\Big)^{p} = \Big\|\frac{|\langle \dot{W}, \psi^n_y \rangle|}{2^{-n(\alpha - 1) -\frac{n}{2}}}\Big\|_{\ell^p_n}^p.
  \end{equation*}
  Hence, what we finally obtained is
  \begin{align*}
    &\left\lVert \Big\| \sup_{\eta \in \mathcal{B}^r}\frac{|\sum_{n \le n_0}\sum_{y \in \Lambda_n}\langle \mathcal{R}f - \Pi_xf(x), \psi^n_y \rangle \langle \psi^n_y, \eta^\lambda_x \rangle |}{\lambda^\gamma} \Big\|_{L^{\frac{p}{2}}(\d x)} \right\rVert_{L^{\frac{p}{2}}_{n_0}(\d \lambda)} \\
    &\qquad\quad\lesssim \sum_{n \le n_0} 2^{(n_0-n)\gamma}\int_{B(0,C^\prime 2^{-n})}2^{n} \Big\|\frac{|Y_{x+h} - Y_x|}{|h|^{\alpha}}\Big\|_{L^p(\d x)} \Big\|\frac{|\langle \dot{W}, \psi^n_y \rangle|}{2^{-n(\alpha - 1) -\frac{n}{2}}}\Big\|_{\ell^p_n} \dd h.
  \end{align*}
  As a consequence, the $\ell^{\frac{p}{2}}(n_0 \in \N_0)$-norm of~\eqref{eq:the sum up to n0} is bounded by
  \begin{align*}
    \Big( &\sum_{n_0 \ge 0} \Big(\sum_{n \le n_0} 2^{(n_0-n)\gamma}\int_{B(0,C^\prime 2^{-n})}2^{n} \Big\|\frac{|Y_{x+h} - Y_x|}{|h|^{\alpha}}\Big\|_{L^p(\d x)} \Big\|\frac{|\langle \dot{W}, \psi^n_y \rangle|}{2^{-n(\alpha - 1) -\frac{n}{2}}}\Big\|_{\ell^p_n} \dd h \Big)^{\frac{p}{2}}\Big)^{\frac{2}{p}} \\
    &\lesssim \Big( \sum_{n_0 \ge 0} \sum_{n \le n_0} 2^{(n_0-n)\gamma}\Big(\int_{B(0,C^\prime 2^{-n})}2^{n} \Big\|\frac{|Y_{x+h} - Y_x|}{|h|^{\alpha}}\Big\|_{L^p(\d x)} \Big\|\frac{|\langle \dot{W}, \psi^n_y \rangle|}{2^{-n(\alpha - 1) -\frac{n}{2}}}\Big\|_{\ell^p_n} \dd h \Big)^{\frac{p}{2}} \Big)^{\frac{2}{p}} \\
    &\lesssim \Big( \sum_{n \ge 0} \Big(\int_{B(0,C^\prime 2^{-n})}2^{n} \Big(\Big\|\frac{|Y_{x+h} - Y_x|}{|h|^{\alpha}}\Big\|_{L^p(\d x)}\Big)^{\frac{p}{2}} \dd h\Big)\Big(\Big\|\frac{|\langle \dot{W}, \psi^n_y \rangle|}{2^{-n(\alpha - 1) -\frac{n}{2}}}\Big\|_{\ell^p_n}\Big)^{\frac{p}{2}} \Big)^{\frac{2}{p}} \\
    &\lesssim \Big(\Big( \sum_{n \ge 0} \Big(\int_{B(0,C^\prime 2^{-n})}2^{n} \Big(\Big\|\frac{|Y_{x+h} - Y_x|}{|h|^{\alpha}}\Big\|_{L^p(\d x)}\Big)^{\frac{p}{2}} \dd h\Big)^2\Big)^{\frac{1}{2}}\Big(\sum_{n \ge 0} \Big\|\frac{|\langle \dot{W}, \psi^n_y \rangle|}{2^{-n(\alpha - 1) -\frac{n}{2}}}\Big\|_{\ell^p_n}^{p}\Big)^{\frac{1}{2}} \Big)^{\frac{2}{p}} \\ 
    &\lesssim \Big( \sum_{n \ge 0} \int_{B(0,C^\prime 2^{-n})}2^{n} \Big\|\frac{|Y_{x+h} - Y_x|}{|h|^{\alpha}}\Big\|_{L^p(\d x)}^{p} \dd h\Big)^{\frac{1}{p}}\Big(\sum_{n \ge 0} \Big\|\frac{|\langle \dot{W}, \psi^n_y \rangle|}{2^{-n(\alpha - 1) -\frac{n}{2}}}\Big\|_{\ell^p_n}^{p}\Big)^{\frac{1}{p}} \\
    &\lesssim \Big(\int_{B(0,C^\prime)} \Big\|\frac{|Y_{x+h} - Y_x|}{|h|^{\alpha}}\Big\|_{L^p(\d x)}^{p}\, \frac{\d h}{|h|}\Big)^{\frac{1}{p}}\Big(\sum_{n \ge 0} \Big\|\frac{|\langle \dot{W}, \psi^n_y \rangle|}{2^{-n(\alpha - 1) -\frac{n}{2}}}\Big\|_{\ell^p_n}^{p}\Big)^{\frac{1}{p}} \\
    &\lesssim \vertiii{f}_{\gamma,p,p} \|\Pi\|_p,
  \end{align*}
  where we used Jensen's inequality for the finite discrete measure $n \in \{0,\dots,n_0\} \mapsto 2^{(n_0-n)\gamma}$ (as $\gamma = 2\alpha -1 <0$) in the second line, Jensen's inequality for the finite measure  $2^n\dd h$ on ${B(0,C^\prime2^{-n})}$ in the third line, H\"older's inequality of the type $\sum |a_nb_n| \le (\sum a_n^2)^{\frac{1}{2}}(\sum b_n^2)^{\frac{1}{2}}$ in the fourth line and again Jensen's inequality for $2^n\dd h$ on ${B(0,C^\prime2^{-n})}$ in the sixth line. We also note that $\Big(\int_{B(0,C^\prime)} \|\frac{|Y_{x+h} - Y_x|}{|h|^{\alpha}}\|_{L^p(\d x)}^{p} \frac{\dd h}{|h|}\Big)^{1/p}$ is the translation bound of the modelled distribution $f$ (so that it can be controlled by $\vertiii{f}_{\gamma,p,p}$) and by \cite[Proposition~2.4]{Hairer2017} the term
  \begin{equation*}
    \Big(\sum_{n \ge 0} \Big\|\frac{|\langle \dot{W}, \psi^n_y \rangle|}{2^{-n(\alpha - 1) -\frac{n}{2}}}\Big\|_{\ell^p_n}^{p}\Big)^{\frac{1}{p}}
  \end{equation*}
  is an equivalent Sobolev norm of $\dot{W} \in W^{\alpha - 1}_p$ which is also the norm of $\Pi$ in the sense of Definition~\ref{def:Sobolev model}.\smallskip
  
  \textit{Step~2:} Now we turn to the Term~\eqref{eq:the sum up after n0}:
  \begin{equation*}
    \left\lVert \Big\| \sup_{\eta \in \mathcal{B}^r}\frac{|\sum_{n > n_0}\sum_{y \in \Lambda_n}\langle \mathcal{R}f - \Pi_xf(x), \psi^n_y \rangle \langle \psi^n_y, \eta^\lambda_x \rangle |}{\lambda^\gamma} \Big\|_{L^{\frac{p}{2}}(\d x)} \right\rVert_{L^{\frac{p}{2}}_{n_0}(\d \lambda)}.
  \end{equation*}
  For $\lambda \in (2^{-n_0-1},2^{-n_0}]$ and $n > n_0$, we have
  \begin{equation*}
    |\langle \psi^n_y , \eta^{\lambda}_x \rangle| \lesssim 2^{-\frac{n}{2} - rn}2^{n_0(1+r)}
  \end{equation*}
  uniformly over all $y \in \Lambda_n$, $\eta \in \mathcal{B}^r$, $x \in \R$ and $n > n_0$. Moreover, this inner product can make contributions only when $|y - x| \le C2^{-n_0}$ for some constant~$C$. Hence, combining this with Estimate~\eqref{eq:Bound for the inner product against single mother wavalet} we get
  \begin{align*}
    &\left\lVert \Big\| \sup_{\eta \in \mathcal{B}^r}\frac{|\sum_{n > n_0}\sum_{y \in \Lambda_n}\langle \mathcal{R}f - \Pi_xf(x), \psi^n_y \rangle \langle \psi^n_y, \eta^\lambda_x \rangle |}{\lambda^\gamma} \Big\|_{L^{\frac{p}{2}}(\d x)} \right\rVert_{L^{\frac{p}{2}}_{n_0}(\d \lambda)}  \\
    &\qquad\lesssim \sum_{n > n_0} 2^{(n_0-n)(r + \alpha - 1)}\Big\|\sum_{y \in \Lambda_n, |y - x| \le C2^{-n_0}}\int_{B(y,2^{-n})}2^{n_0} \frac{|Y_z - Y_x|}{2^{-n_0\alpha}}\frac{|\langle \dot{W}, \psi^n_y \rangle|}{2^{-n(\alpha - 1) - \frac{n}{2}}} \dd z \Big\|_{L^{\frac{p}{2}}(\d x)}.
  \end{align*}
  Since 
  \begin{align*}
    \sum_{y \in \Lambda_n, |y - x| \le C2^{-n_0}} & \int_{B(y,2^{-n})}2^{n_0} \frac{|Y_z - Y_x|}{2^{-n_0\alpha}}\frac{|\langle \dot{W}, \psi^n_y \rangle|}{2^{-n(\alpha - 1) - \frac{n}{2}}} \dd z  \\
    &\lesssim  \int_{B(0,C^\prime 2^{-n_0})}2^{n_0}\frac{|Y_{x+h} - Y_x|}{2^{-n_0\alpha}} \dd h
    \,\Big(\max_{y \in \Lambda_n, |y - x| \le C2^{-n_0}}\frac{|\langle \dot{W}, \psi^n_y \rangle|}{2^{-n(\alpha - 1) - \frac{n}{2}}}\Big),
  \end{align*}
  holds for each $n > n_0$, we can deduce that
  \begin{align*}
    &\Big\|\sum_{y \in \Lambda_n, |y - x| \le C2^{-n_0}}\int_{B(y,2^{-n})}2^{n_0} \frac{|Y_z - Y_x|}{2^{-n_0\alpha}}\frac{|\langle \dot{W}, \psi^n_y \rangle|}{2^{-n(\alpha - 1) - \frac{n}{2}}} \dd z \Big\|_{L^{\frac{p}{2}}(\d x)} \\
    &\qquad\lesssim \int_{B(0,C^\prime 2^{-n_0})}2^{n_0}\Big(\int_{\R} \Big(\frac{|Y_{x+h} - Y_x|}{2^{-n_0\alpha}} \Big)^{\frac{p}{2}}\Big(\max_{y \in \Lambda_n, |y - x| \le C2^{-n_0}}\frac{|\langle \dot{W}, \psi^n_y \rangle|}{2^{-n(\alpha - 1) - \frac{n}{2}}}\Big)^{\frac{p}{2}} \dd x \Big)^{\frac{2}{p}} \dd h,
  \end{align*}
  where we used  the Minkowski's integral inequality as in Step~1. Then by H\"older's inequality, we find that
  \begin{align*}
    &\Big(\int_{\R} \Big(\frac{|Y_{x+h} - Y_x|}{2^{-n_0\alpha}} \Big)^{\frac{p}{2}}\Big(\max_{y \in \Lambda_n, |y - x| \le C2^{-n_0}}\frac{|\langle \dot{W}, \psi^n_y \rangle|}{2^{-n(\alpha - 1) - \frac{n}{2}}}\Big)^{\frac{p}{2}} \dd x \Big)^{\frac{2}{p}} \\
    &\qquad\quad\lesssim \Big(\int_{\R}\Big(\frac{|Y_{x+h} - Y_x|}{2^{-n_0\alpha}} \Big)^{p} \dd x\Big)^{\frac{1}{p}}\Big(\int_{\R}  \Big(\max_{y \in \Lambda_n, |y - x| \le C2^{-n_0}}\frac{|\langle \dot{W}, \psi^n_y \rangle|}{2^{-n(\alpha - 1) - \frac{n}{2}}}\Big)^{p} \dd x  \Big)^{\frac{1}{p}}.
  \end{align*}
  Next we consider the integral $\int_{\R}  \Big(\max_{y \in \Lambda_n, |y - x| \le C2^{-n_0}}\frac{|\langle \dot{W}, \psi^n_y \rangle|}{2^{-n(\alpha - 1) - \frac{n}{2}}}\Big)^{p} \dd x$. As before, we rewrite it as
  \begin{equation*}
    \sum_{z \in \Lambda_n} \int_{x \in B(z,2^{-n-1})} \Big(\max_{y \in \Lambda_n, |y - x| \le C2^{-n_0}}\frac{|\langle \dot{W}, \psi^n_y \rangle|}{2^{-n(\alpha - 1) - \frac{n}{2}}}\Big)^{p} \dd x,
  \end{equation*}
  and observe the estimate
  \begin{align*}
    \sum_{z \in \Lambda_n} \int_{x \in B(z,2^{-n-1})} \Big(\max_{y \in \Lambda_n, |y - x| \le C2^{-n_0}}
    &\frac{|\langle \dot{W}, \psi^n_y \rangle|}{2^{-n(\alpha - 1) - \frac{n}{2}}}\Big)^{p} \dd x \\
    &\le \sum_{z \in \Lambda_n} 2^{-n} \sum_{y \in \Lambda_n, |y - z| \le C^\prime 2^{-n_0}}\Big(\frac{|\langle \dot{W}, \psi^n_y \rangle|}{2^{-n(\alpha - 1) - \frac{n}{2}}}\Big)^{p}
  \end{align*}
  for some constant $C^\prime$. Since the number of $y \in \Lambda_n$ such that $|y - z| \le C^\prime 2^{-n_0}$ is of order $2^{n-n_0}$ for $n > n_0$ uniformly over all $z \in \Lambda_n$, we count every $y \in \Lambda_n$ for (a multiple of) $2^{n-n_0}$ times. This implies that
  \begin{equation*}
    \sum_{z \in \Lambda_n} 2^{-n} \sum_{y \in \Lambda_n, |y - z| \le C^\prime 2^{-n_0}}\Big(\frac{|\langle \dot{W}, \psi^n_y \rangle|}{2^{-n(\alpha - 1) - \frac{n}{2}}}\Big)^{p}
    \lesssim 2^{n-n_0} \sum_{y \in \Lambda_n}2^{-n}\Big(\frac{|\langle \dot{W}, \psi^n_y \rangle|}{2^{-n(\alpha - 1) - \frac{n}{2}}}\Big)^{p}
  \end{equation*}
  and hence
  \begin{equation*}
    \Big(\int_{\R}  \Big(\max_{y \in \Lambda_n, |y - x| \le C2^{-n_0}}\frac{|\langle \dot{W}, \psi^n_y \rangle|}{2^{-n(\alpha - 1) - \frac{n}{2}}}\Big)^{p} \dd x  \Big)^{\frac{1}{p}} \lesssim 2^{(n-n_0)\frac{1}{p}} \Big\|\frac{|\langle \dot{W}, \psi^n_y \rangle|}{2^{-n(\alpha - 1) - \frac{n}{2}}}\Big\|_{\ell^p_n}.
  \end{equation*}
  So, finally we obtain that
  \begin{align*}
    &\left\lVert \Big\| \sup_{\eta \in \mathcal{B}^r}\frac{|\sum_{n > n_0}\sum_{y \in \Lambda_n}\langle \mathcal{R}f - \Pi_xf(x), \psi^n_y \rangle \langle \psi^n_y, \eta^\lambda_x \rangle |}{\lambda^\gamma} \Big\|_{L^{\frac{p}{2}}(\d x)} \right\rVert_{L^{\frac{p}{2}}_{n_0}(\d \lambda)}  \\
    &\quad\lesssim \sum_{n > n_0} 2^{(n_0-n)(r + \alpha - 1 - \frac{1}{p})}\Big(\int_{h \in B(0,C^\prime 2^{-n_0})}2^{n_0} \Big\|\frac{Y_{x+h} - Y_{x}}{|h|^\alpha}\Big\|_{L^p(\d x)} \dd h \Big) \Big(\Big\|\frac{|\langle \dot{W}, \psi^n_y \rangle|}{2^{-n(\alpha - 1) - \frac{n}{2}}}\Big\|_{\ell^p_n} \Big).
  \end{align*}
  Thanks to our choice of~$r$ (that is $r > |\alpha - 1 - \frac{1}{p}|$), for $\theta := r + \alpha - 1 - \frac{1}{p}$ the discrete measure 
  \begin{equation*}
    n \in \{n_0 + 1,\dots\} \mapsto 2^{(n_0 - n)\theta}
  \end{equation*}
  has finite total mass independent of~$n_0$, hence by Jensen's inequality and we can get that
  \begin{align*}
    &\Big( \sum_{n_0 \ge 0} \Big(\sum_{n > n_0} 2^{(n_0-n)\theta}\Big(\int_{h \in B(0,C^\prime 2^{-n_0})}2^{n_0} \Big\|\frac{Y_{x+h} - Y_{x}}{|h|^\alpha}\Big\|_{L^p(\d x)} \dd h \Big) \Big(\Big\|\frac{|\langle \dot{W}, \psi^n_y \rangle|}{2^{-n(\alpha - 1) - \frac{n}{2}}}\Big\|_{\ell^p_n} \Big) \Big)^{\frac{p}{2}} \Big)^{\frac{2}{p}} \\
    &\lesssim \Big(\sum_{n\ge 0} \Big(\Big\|\frac{|\langle \dot{W}, \psi^n_y \rangle|}{2^{-n(\alpha - 1) - \frac{n}{2}}}\Big\|_{\ell^p_n} \Big)^{\frac{p}{2}}\sum_{n_0=0}^n 2^{(n_0-n)\theta}\Big(\int_{h \in B(0,C^\prime 2^{-n_0})}2^{n_0} \Big\|\frac{Y_{x+h} - Y_{x}}{|h|^\alpha}\Big\|_{L^p(\d x)}^{\frac{p}{2}} \dd h \Big) \Big)^{\frac{2}{p}} \\
    &\lesssim  \Big(\sum_{n\ge 0} \Big\|\frac{|\langle \dot{W}, \psi^n_y \rangle|}{2^{-n(\alpha - 1) - \frac{n}{2}}}\Big\|_{\ell^p_n}^{p}\Big)^{\frac{1}{p}}\Big(\sum_{n \ge 0}\Big(\sum_{n_0=0}^n 2^{(n_0-n)\theta}\int_{h \in B(0,C^\prime 2^{-n_0})}2^{n_0} \Big\|\frac{Y_{x+h} - Y_{x}}{|h|^\alpha}\Big\|_{L^p(\d x)}^{\frac{p}{2}} \dd h \Big)^2\Big)^{\frac{1}{p}} \\
    &\lesssim \Big(\sum_{n\ge 0} \Big\|\frac{|\langle \dot{W}, \psi^n_y \rangle|}{2^{-n(\alpha - 1) - \frac{n}{2}}}\Big\|_{\ell^p_n}^{p}\Big)^{\frac{1}{p}}\Big(\sum_{n \ge 0}\sum_{n_0 = 0}^n 2^{(n_0-n)\theta}\int_{h \in B(0,C^\prime 2^{-n_0})}2^{n_0} \Big\|\frac{Y_{x+h} - Y_{x}}{|h|^\alpha}\Big\|_{L^p(\d x)}^{p} \dd h \Big)^{\frac{1}{p}} \\
    &\lesssim \Big(\sum_{n\ge 0} \Big\|\frac{|\langle \dot{W}, \psi^n_y \rangle|}{2^{-n(\alpha - 1) - \frac{n}{2}}}\Big\|_{\ell^p_n}^{p}\Big)^{\frac{1}{p}} \Big(\int_{h \in B(0,C^\prime)} \Big\|\frac{Y_{x+h} - Y_{x}}{|h|^\alpha}\Big\|_{L^p(\d x)}^{p} \frac{\dd h}{|h|}\Big)^{\frac{1}{p}} \\
    &\lesssim \vertiii{f}_{\gamma,p,p}\|\Pi\|_{p},
  \end{align*}
  where we used Jensen's inequality and H\"older's inequality in the same way as in Step~1. Hence, we showed that the $\ell^{\frac{p}{2}}_{n_0}$-norm of~\eqref{eq:the sum up after n0} is also bounded by $\vertiii{f}_{\gamma,p,p}\|\Pi\|_{p}$, as claimed.
\end{proof}
 
With these two lemmas at hand, we are in a position to prove Theorem~\ref{thm:Sobolev path lift}, which ensures the existence of a Sobolev rough path lift above a Sobolev path.

\begin{proof}[Proof of Theorem~\ref{thm:Sobolev path lift}]
  Without loss of generality we set $T=1$ keeping in mind that there is a smooth transformation between $[0,T]$ and $[0,1]$. Moreover, we set $f_{s,t}:=f_t -f_s$ for the increment of a function $f\colon [0,T]\to \R^d$, where $s,t\in [0,T]$.

  In view of Lemma~\ref{lem:regularity of Rf} and Lemma~\ref{lem:Bound in reconstruction theorem w.r.t. Sobolev model}, there exists a distribution $\dot{Z} := \mathcal{R}\dot{\mathbb{Z}} \in W^{\alpha-1}_p$ such that
  \begin{equation*}
    \left\lVert \Big\| \sup_{\eta \in \mathcal{B}^r} \frac{|\langle \dot{Z} - \Pi_x\dot{\mathbb{Z}}(x), \eta^\lambda_x \rangle|}{\lambda^\gamma} \Big\|_{L^{\frac{p}{2}}(\d x)} \right\rVert_{L^{\frac{p}{2}}_\lambda} \lesssim 1.
  \end{equation*}
  Since $\Pi_x\dot{\mathbb{Z}}(x) = Y_x\dot{W}$ for all $x \in \R$, it holds that
  \begin{equation*}
    \left\lVert \Big\| \sup_{\eta \in \mathcal{B}^r} \frac{|\langle \dot{Z} - Y_x\dot{W}, \eta^\lambda_x \rangle|}{\lambda^\gamma} \Big\|_{L^{\frac{p}{2}}(\d x)} \right\rVert_{L^{\frac{p}{2}}_\lambda} \lesssim 1.
  \end{equation*}
  Then by Fubini's theorem we obtain that
  \begin{align}\label{eq:Sobolev type estimate based on reconstruction theorem}
  \begin{split}
    \left\lVert \Big\| \sup_{\eta \in \mathcal{B}^r} \frac{|\langle \dot{Z} - Y_x\dot{W}, \eta^\lambda_x \rangle|}{\lambda^\gamma} \Big\|_{L^{\frac{p}{2}}(\d x)} \right\rVert_{L^{\frac{p}{2}}_\lambda}^{\frac{p}{2}} 
    &= \int_0^1 \int_{\R} \Big(\sup_{\eta \in \mathcal{B}^r}\frac{|\langle \dot{Z} - Y_x\dot{W}, \eta^\lambda_x \rangle|}{\lambda^\gamma}\Big)^{\frac{p}{2}} \dd x\, \frac{\dd \lambda}{\lambda} \\
    &=\int_{\R} \int_0^1 \sup_{\eta \in \mathcal{B}^r}\frac{|\langle \dot{Z} - Y_x\dot{W}, \eta^\lambda_x \rangle|^{\frac{p}{2}}}{\lambda^{\gamma \frac{p}{2} + 1}}\dd \lambda \dd x \\
    &=\int_{\R} \int_{x}^{x+1} \sup_{\eta \in \mathcal{B}^r}\frac{|\langle \dot{Z} - Y_x\dot{W}, \eta^{(y-x)}_x \rangle|^{\frac{p}{2}}}{(y-x)^{\gamma \frac{p}{2} + 1}}\dd y \dd x,
  \end{split}
  \end{align}
  where in the last equality we used the change-of-variable $\lambda = y - x$ for every $x \in \R$. Now we choose $\eta := \1_{[0,1]}$ such that $\eta^{(y-x)}_x = \frac{1}{y-x}\1_{[x,y]}$ for $y \in (x,x+1]$, and then follow the same arguments in the relevant proof of  \cite[Theorem~4.6]{Brault2019} (more precisely, a straight forward calculation reveals that \cite[Lemma~3.10]{Brault2019} remains valid in the current Sobolev setup, from which one can easily establish the following bound for indicator function)  to show that the bound~\eqref{eq:Sobolev type estimate based on reconstruction theorem} remains valid for this indicator function and consequently
  \begin{equation*}
    \int_{\R} \int_{x}^{x+1} \frac{|\langle \dot{Z} - Y_x\dot{W}, \1_{[x,y]} \rangle|^{\frac{p}{2}}}{(y-x)^{(\gamma+1)\frac{p}{2} + 1}}\dd y \dd x \lesssim \left\lVert \Big\| \sup_{\eta \in \mathcal{B}^r} \frac{|\langle \dot{Z} - \Pi_x\dot{\mathbb{Z}}(x), \eta^\lambda_x \rangle|}{\lambda^\gamma} \Big\|_{L^{\frac{p}{2}}(\d x)} \right\rVert_{L^{\frac{p}{2}}_\lambda}^{\frac{p}{2}} \lesssim 1.
  \end{equation*}
  Since $\dot{Z} \in W^{\alpha-1}_p$, the primitive $Z$ of $\dot{Z}$, which is a distribution in $W^{\alpha}_p$, is continuous due to the classical embedding theorem. Hence, we can immediately check that 
  \begin{equation*}
    \langle \dot{Z} - Y_x\dot{W}, \1_{[x,y]} \rangle = Z_{x,y} - Y_x W_{x,y}
  \end{equation*}
  (by approximation, of course) and conclude
  \begin{equation}\label{eq:estimate from reconstruction theorem}
    \int_0^1 \int_{x}^{1} \frac{|Z_{x,y} - Y_x W_{x,y}|^{\frac{p}{2}}}{(y-x)^{\alpha p + 1}}\dd y \dd x \lesssim 1.
  \end{equation}
  Now we define $\mathbb{X}^{1,2}_{x,y} := Z_{x,y} - Y_x W_{x,y} $ on $\Delta := \{(x,y) \in \R^2 : 0 \le x \le y \le 1\}$, Estimate~\eqref{eq:estimate from reconstruction theorem} can be written as
  \begin{equation*}
    \iint_{\Delta}\frac{|\mathbb{X}^{1,2}_{x,y}|^{\frac{p}{2}}}{(y-x)^{\alpha p + 1}}\dd y \dd x \lesssim 1.
  \end{equation*}
  Similarly, we can obtain the same bound for $\mathbb{X}^{2,1}_{x,y} := Z_{x,y} - W_xY_{x,y}$, where now $Z$ denotes the primitive of $\mathcal{R}\dot{\mathbb{Z}}$ obtained from Lemma~\ref{lem:regularity of Rf} for the same regularity structure as before with $\dot{\mathbb{W}}$ replaced by $\dot{\mathbb{Y}}$ such that $\Pi_x\dot{\mathbb{Y}} = \dot{Y}$ and $\dot{\mathbb{Z}}(x) := W_x\dot{\mathbb{Y}}$. The notations $\mathbb{X}^{1,1}_{x,y} = Z_{x,y} - Y_xY_{x,y}$ and $\mathbb{X}^{2,2}_{x,y} = Z_{x,y} - W_x W_{x,y}$ are then self-explanatory (although we use $Z$ to denote different functions). Let $\mathbb{X}_{x,y} := \mathbb{X}^{i,j}_{x,y}$ for $x,y \in I$ and $i,j=1,2$, then Bound~\eqref{eq:estimate from reconstruction theorem} guarantees that
  \begin{equation}\label{eq:regularity second level}
    \int_0^1 \int_{0}^{1} \frac{|X_{x,y}|^p + |\mathbb{X}_{x,y}|^{\frac{p}{2}}}{|y-x|^{\alpha p + 1}}\dd y \dd x \lesssim 1.
  \end{equation}
  Moreover, we can immediately check that $\mathbb{X}$ satisfies Chen's relation by construction. Now, we define $F = (F^{i,j})_{i,j = 1,2}$, which is a continuous paths taking value in $\R^2 \otimes \R^2$ such that $F^{i,j}_{x} = F^{j,i}_{x} = \frac{1}{2}(X^i_{0,x}X^j_{0,x} - \mathbb{X}^{i,j}_{0,x} - \mathbb{X}^{j,i}_{0,x})$ for $x \in I$ and $i,j=1,2$ with $X^1= Y$ and $X^2 = W$. Then it is easy to check that $F^{i,j}_{x,y} = \frac{1}{2}(X^i_{x,y}X^j_{x,y} - \mathbb{X}^{i,j}_{x,y} - \mathbb{X}^{j,i}_{x,y})$ for any $(x,y) \in \Delta$ and $\X_{x} := (X_{0,x}, \mathbb{X}_{0,x} + F_{x})$ takes value in $G^2(\R^2)$ and
  \begin{equation}\label{eq:regularity of lift}
    \int_0^1 \int_{x}^{1} \frac{d_{cc}(\X_x,\X_y)^p}{(y-x)^{\alpha p + 1}}\dd y \dd x \lesssim 1,
  \end{equation}
  which indeed means that $\X \in W^{\alpha}_p([0,T];G^2(\R^2))$.
\end{proof}

\begin{remark}\label{rmk:new bounds}
  In the proof of Theorem~\ref{thm:Sobolev path lift} we have seen that the new Bound~\eqref{eq:Bound of reconstruction operator for Besov modelled distribution} was essential to obtain the Sobolev regularity of the rough path lift, see~\eqref{eq:regularity of lift}. This would not have been possible with the original bounds (cf.~\cite[Theorem~3.1]{Hairer2017} and \cite[Theorem~2.11]{Liu2016}) of the reconstruction operator relying on (standard) modelled distributions with H{\"o}lder bounds, which read in our case as
  \begin{equation*}
    \left\lVert \Big\| \sup_{\eta \in \mathcal{B}^r} \frac{|\langle \mathcal{R}\dot{\mathbb{Z}} - \Pi_x\dot{\mathbb{Z}}(x), \eta^\lambda_x \rangle|}{\lambda^\gamma} \Big\|_{L^{p}(\d x)} \right\rVert_{L^{p}_\lambda} \lesssim \|\Pi\|_{p}(1 + \|\Gamma\|_{p})\vertiii{\dot{\mathbb{Z}}}_{\gamma,p,p}.
  \end{equation*}
  This bound leads only to the regularity estimate 
  \begin{equation}\label{eq:remark regularity}
    \int_0^1 \int_{0}^{1} \frac{|X_{x,y}|^p + |\mathbb{X}_{x,y}|^{p}}{|y-x|^{\alpha p + 1}}\dd y \dd x \lesssim 1
  \end{equation}
  and not to the required Estimate~\eqref{eq:regularity second level}. Note, while the Estimate~\eqref{eq:remark regularity} gives the ``right'' regularity parameter of the second order term~$\mathbb{X}$, the integrability parameter is not the required one (here: $p$ instead of $p/2$).
\end{remark}

\begin{remark}
  As we have seen, lifting a path to a rough path based on Hairer's theory of regularity theory requires essentially the reconstruction operator for modelled distributions with negative regularity, which leads to the expected non-uniqueness of the rough path lift, cf. \cite{Hairer2014,Caravenna2020,Liu2016}.
\end{remark}

\subsection{Continuity of the rough path lifting map}

Let us conclude by showing that the method used to construct rough paths via Hairer's reconstruction theorem actually provides a \emph{continuous} way to lift $\R^d$-valued Sobolev paths to Sobolev rough paths of the same regularity. 

For this purpose the distance between two elements $\X^1$ and $\X^2$ in $W^\alpha_p([0,T];G^2(\R^d))$ will be measured with respect to the inhomogeneous Sobolev metric. The \textit{inhomogeneous Sobolev metric} $\rho_{W^{\alpha}_p}$ is defined by
\begin{equation*}
  \rho_{W^{\alpha}_p}(\X^1,\X^2) := \sum_{k=1,2}\rho^{(k)}_{W_p^{\alpha}}(\X^1,\X^2)
\end{equation*}
and for each $k$,
\begin{equation*}
  \rho^{(k)}_{W_p^{\alpha}}(\X^1,\X^2) := \Big(\int_0^{T}\int_0^T \frac{|\pi_k(\X^1_{s,t} - \X^2_{s,t})|^{p/k}}{|t - s|^{\alpha p + 1}} \dd s \dd t\Big)^{k/p}.
\end{equation*}
The inhomogeneous metrics play an important role in the theory of rough differential equations as, for instance, the It{\^o}--Lyons map is continuous with respect to inhomogeneous metrics, cf. \cite{Liu2021}. For a general discussion of inhomogeneous norms and distances in the rough path theory we refer to \cite[Chapter~8]{Friz2010}.

The next theorem is a generalization of \cite[Proposition~13.23]{Friz2014} (see also \cite[Theorem~4.6]{Brault2019}) from H{\"o}lder spaces to Sobolev spaces.

\begin{theorem}\label{thm:continuity of Sobolev rough path lift}
  Let $\alpha \in (1/3, 1/2)$ and $p \in (1,+\infty]$ be such that $\alpha > 1/p$. Then, there exists a map 
  \begin{align*}
    L\colon W^\alpha_p([0,T];\R^d) \to W^\alpha_p([0,T];G^2(\R^d)),\quad \text{via}\quad
    X \mapsto L(X) =: \X,
  \end{align*}
  such that $\X$ is a Sobolev rough path lift of $X$ and $L$ is locally Lipschitz continuous with respect to the inhomogeneous Sobolev metric $\rho_{W^{\alpha}_p}$.
\end{theorem}

\begin{proof}
  Without loss of generality we again assume that $d = 2$ and $T=1$. Throughout the whole proof, we fix a wavelet analysis with father wavelet $\varphi$ and mother wavelet $\psi$ in $\mathcal{C}^r_0$ with $r > |\alpha - 1 -\frac{1}{p}|$, which satisfy the desired properties (1)-(5) for wavelet analysis introduced in Subsection~\ref{subsec:lift via regularity structures}. 
  
  First, let us briefly summarize how to get a rough path lift by using Theorem~\ref{thm:Sobolev path lift}: Let $X = (Y,W) \in W^\alpha_p([0,T];\R^d)$ be given. As we have shown in the proof of Theorem~\ref{thm:Sobolev path lift}, if we apply the Sobolev model introduced in Example~\ref{ex:basic Sobolev model} and define $f(t) := Y_t \dot{\mathbb{W}}$, then it holds that $f \in \mathcal{D}^\gamma_{p,p}$ with $\gamma = 2\alpha - 1$, and the distribution $\mathcal{R}f \in W^{\alpha-1}_p$ defined as in~\eqref{eq:reconstrcution operator for Besov model} satisfies Bound~\eqref{eq:Bound of reconstruction operator for Besov modelled distribution}. Furthermore, let $Z \in W^\alpha_p$ be the primitive of $\mathcal{R}f$, then $\mathbb{X}^{1,2}_{s,t} := Z_{s,t} - Y_sW_{s,t}$ for $s,t \in [0,1]$ satisfies that
  \begin{equation*}
    \int_0^1\int_0^1 \frac{|\mathbb{X}^{1,2}_{s,t}|^{p/2}}{|t-s|^{\alpha p + 1}} \dd s \dd t \lesssim 1.
  \end{equation*}
  Using the same way we can obtain other components $\mathbb{X}^{1,1}$, $\mathbb{X}^{2,1}$ and $\mathbb{X}^{2,2}$ such that $\X_{t} := (X_{0,t}, \mathbb{X}_{0,t} + F_{t})$ is a rough path in $W^\alpha_p(G^2(\R^2))$ over $X$, where $F^{i,j}_{t} = F^{j,i}_{t} = \frac{1}{2}(X^i_{0,t}X^j_{0,t} - \mathbb{X}^{i,j}_{0,t} - \mathbb{X}^{j,i}_{0,t})$ for $t \in [0,1]$ and $i,j=1,2$ with $X^1= Y$ and $X^2 = W$.
  
  Now we set $L(X) := \X$ and thus the map $L \colon W^\alpha_p([0,T];\R^d) \to W^\alpha_p([0,T];G^2(\R^d))$ is well-defined. It only remains to show that $L$ is locally Lipschitz continuous with respect to the metric $\rho_{W^{\alpha}_p}$.\smallskip
  
  \textit{Step~1:} Fix an $X = (Y,W)$ in $W^\alpha_p([0,T];\R^2)$ and let $\tilde{X} = (\tilde{Y},\tilde{W})$ be another element in $W^\alpha_p(\R^2)$. Let $\dot{\tilde{W}}$ be the derivative of $\tilde{W}$. We define a Sobolev model $(\tilde{\Pi}, \tilde{\Gamma})$ for the regularity structure $(\mathcal{A},\mathcal{T},\mathcal{G})$ given in Example~\ref{ex:basic Sobolev model} as following:
  \begin{equation*}
    \tilde{\Pi}_t(\dot{\mathbb{W}}) := \dot{\tilde{W}},   \quad    \tilde{\Pi}_t(\1) := 1 \in \R,
  \end{equation*}
  and $\tilde{\Gamma}_{s,t} = \text{Id}_{\Tc}$ for all $s,t \in \R$. Note that $(\tilde{\Pi}, \tilde{\Gamma})$ is the model used for constructing rough path lift over $\tilde{X}$. Hence, by defining $g(t) := \tilde{Y}_t \dot{\mathbb{W}}$, we have $g \in \tilde{D}^\gamma_{p,p}$, where $\tilde{D}^\gamma_{p,p}$ is the space of modelled distributions associated to $(\tilde{\Pi}, \tilde{\Gamma})$. Then, as we stated above, if $\tilde{\mathcal{R}}g$ is defined as in~\eqref{eq:reconstrcution operator for Besov model} by changing $f$ to $g$, $\Pi$ to $\tilde{\Pi}$ and using the same wavelet basis, its primitive $\tilde{Z} \in W^\alpha_p$ satisfies that $\tilde{Z}_{s,t} - \tilde{Y}_s\tilde{W}_{s,t} = \tilde{\mathbb{X}}^{1,2}_{s,t}$, where $\tilde{\mathbb{X}} = (\tilde{\mathbb{X}}^{i,j})_{i,j=1,2}$ is the second level component of $L(\tilde{X}) = \tilde{\X}$ up to an addition of the function $\tilde{F}$ which is the counterpart of the function $F$ defined as above with $\tilde{X}$ replacing $X$.\smallskip
  
  \textit{Step~2:} Next we will show that
  \begin{equation}\label{eq:estimates for comparing the second levels of two Besov rough paths}
    \Big(\int_0^1\int_0^1 \frac{|\mathbb{X}^{1,2}_{s,t} - \tilde{\mathbb{X}}^{1,2}_{s,t}|^{p/2}}{|t -s |^{\alpha p + 1}} \dd s \dd t\Big)^{\frac{2}{p}} \lesssim_{X,\tilde{X}}\| X - \tilde{X} \|_{W^\alpha_p}.
  \end{equation}
  To this end, first of all we note that in view of the definitions of $\mathcal{R}f$ and $\tilde{\mathcal{R}}g$ (see~\eqref{eq:reconstrcution operator for Besov model}), for every $n \ge 0$, $x \in \R$ and $y \in \Lambda_n$, it holds that 
  \begin{align*}
    &\langle \mathcal{R}f - \tilde{\mathcal{R}}g - \Pi_xf(x) + \tilde{\Pi}_xg(x), \psi^n_y \rangle \\
    & \qquad\quad = \int_{z \in B(y, 2^{-n})} 2^n\langle \Pi_y(f(z) - f(x) - g(z) + g(x)), \psi^n_y \rangle \dd z\\
    & \qquad\quad\quad + \int_{z \in B(y, 2^{-n})} 2^n\langle (\Pi_y - \tilde{\Pi}_y)(g(z) - g(x)), \psi^n_y \rangle \dd z.
  \end{align*}
  Then, by the definitions of the models $(\Pi,\Gamma)$ and $(\tilde{\Pi},\tilde{\Gamma})$ as well as the constructions of the modelled distributions~$f$ and~$g$, we have 
  \begin{equation*}
    \Pi_y(f(z) - f(x) - g(z) + g(x)) = (Y_z - Y_x - \tilde{Y}_z + \tilde{Y}_x) \dot{W}
  \end{equation*}
  and
  \begin{equation*}
    (\Pi_y - \tilde{\Pi}_y)(g(z) - g(x)) = (\tilde{Y}_z - \tilde{Y}_x)(\dot{W} - \dot{\tilde{W}}).
  \end{equation*}
  Hence, we obtain that
  \begin{align*}
    \langle \mathcal{R}f - \tilde{\mathcal{R}}g - \Pi_xf(x) + \tilde{\Pi}_xg(x), \psi^n_y \rangle
    = &\int_{z \in B(y, 2^{-n})} 2^n\langle (Y_z - Y_x - \tilde{Y}_z + \tilde{Y}_x) \dot{W}, \psi^n_y \rangle \dd z \\
    & \quad +  \int_{z \in B(y, 2^{-n})} 2^n\langle (\tilde{Y}_z - \tilde{Y}_x)(\dot{W} - \dot{\tilde{W}}), \psi^n_y \rangle \dd z.
  \end{align*}
  Then, following the arguments used in the proof of Lemma~\ref{lem:Bound in reconstruction theorem w.r.t. Sobolev model} we can derive that
  \begin{align*}
    &\left\lVert \Big\| \sup_{\eta \in \mathcal{B}^r} \frac{|\langle \mathcal{R}f - \tilde{\mathcal{R}}g - \Pi_xf(x) + \tilde{\Pi}_xg(x), \eta^\lambda_x \rangle|}{\lambda^\gamma} \Big\|_{L^{\frac{p}{2}}(\d x)} 
    \right\rVert_{L^{\frac{p}{2}}_\lambda} \\
    &\qquad \qquad\qquad\qquad \qquad\qquad\qquad \qquad\qquad
    \lesssim \|\Pi\|_{p}\vertiii{f - g}_{\gamma,p,p} + \|\Pi - \tilde{\Pi}\|_p\vertiii{g}_{\gamma,p,p}.
  \end{align*}
  Since $\|\Pi\|_p = \| \dot{W} \|_{W^{\alpha - 1}_p}$, $\|\Pi - \tilde{\Pi}\|_p = \|\dot{W} - \dot{\tilde{W}}\|_{W^{\alpha-1}_p}$, $\vertiii{g}_{\gamma,p,p} \lesssim \|\tilde{Y}\|_{W^\alpha_p}$ and $\vertiii{f - g}_{\gamma,p,p} \lesssim \|Y - \tilde{Y}\|_{W^\alpha_p}$, the above inequality can be written as
  \begin{align}\label{eq:First bound for two Besov rough paths}
  \begin{split}
    &\left\lVert \Big\| \sup_{\eta \in \mathcal{B}^r} \frac{|\langle \mathcal{R}f - \tilde{\mathcal{R}}g - \Pi_xf(x) + \tilde{\Pi}_xg(x), \eta^\lambda_x \rangle|}{\lambda^\gamma} \Big\|_{L^{\frac{p}{2}}(\d x)} \right\rVert_{L^{\frac{p}{2}}_\lambda}  \\
    &\qquad\qquad\qquad\lesssim \| \dot{W} \|_{W^{\alpha - 1}_p}\|Y - \tilde{Y}\|_{W^\alpha_p} + \|\dot{W} - \dot{\tilde{W}}\|_{W^{\alpha-1}_p}\|\tilde{Y}\|_{W^\alpha_p}  \\
    &\qquad\qquad\qquad
    \lesssim_{X,\tilde{X}} \|X - \tilde{X}\|_{W^\alpha_p},
  \end{split}
  \end{align}
  where in the third line we used \cite[Theorem~2.3.8]{Triebel2010}.
  
  Now, invoking that $\mathcal{R}f - \tilde{\mathcal{R}}g - \Pi_xf(x) + \tilde{\Pi}_xg(x) = \dot{Z} - Y_x\dot{W} - (\dot{\tilde{Z}} - \tilde{Y}_x\dot{\tilde{W}})$, we can apply the same argument as for establishing~\eqref{eq:estimate from reconstruction theorem} to the Estimate~\eqref{eq:First bound for two Besov rough paths} to get that
  \begin{equation*}
    \Big(\int_0^1\int_0^1 \frac{|Z_{s,t} - Y_sW_{s,t} - (\tilde{Z}_{s,t} - \tilde{Y}_s\tilde{W}_{s,t})|^{\frac{p}{2}}}{|t-s|^{\alpha p +1}} \dd s \dd t \Big)^{\frac{2}{p}}\lesssim_{X,\tilde{X}} \|X - \tilde{X}\|_{W^\alpha_p}.
  \end{equation*}
  Since $Z_{s,t} - Y_sW_{s,t} = \mathbb{X}^{1,2}_{s,t}$ and $\tilde{Z}_{s,t} - \tilde{Y}_s\tilde{W}_{s,t} = \tilde{\mathbb{X}}_{s,t}$, Estimate~\eqref{eq:estimates for comparing the second levels of two Besov rough paths} has been established.\smallskip
  
  \textit{Step~3:} The estimate from Step~2 gives that
  \begin{equation*}
    \Big(\int_0^1\int_0^1 \frac{|\mathbb{X}_{s,t} - \tilde{\mathbb{X}}_{s,t}|^{p/2}}{|t -s |^{\alpha p + 1}} \dd s \dd t\Big)^{\frac{2}{p}} \lesssim_{X,\tilde{X}}\| X - \tilde{X} \|_{W^\alpha_p},
  \end{equation*}
  which in turn implies that the same bound also holds true for $\Big(\int_0^1\int_0^1 \frac{|F_{s,t} - \tilde{F}_{s,t}|^{p/2}}{|t -s |^{\alpha p + 1}} \dd s \dd t\Big)^{\frac{2}{p}}$ by invoking the definitions of $F$ and $\tilde{F}$. Hence, noting that 
  \begin{equation*}
    \pi_2(\X_{s,t} - \tilde{\X}_{s,t}) = \mathbb{X}_{s,t} - \tilde{\mathbb{X}}_{s,t} + F_{s,t} - \tilde{F}_{s,t}
  \end{equation*}
  we can deduce that
  \begin{equation*}
    \rho^{(2)}_{W^{\alpha}_p}(L(X),L(\tilde{X})) = \Big(\int_0^1\int_0^1 \frac{|\mathbb{X}_{s,t} - \tilde{\mathbb{X}}_{s,t} - F_{s,t} + \tilde{F}_{s,t}|^{p/2}}{|t -s |^{\alpha p + 1}} \dd s \dd t\Big)^{\frac{2}{p}} \lesssim_{X,\tilde{X}}\| X - \tilde{X} \|_{W^\alpha_p}.
  \end{equation*}
  Since $\rho^{(1)}_{W^{\alpha}_p}(L(X),L(\tilde{X})) = \| X - \tilde{X} \|_{W^\alpha_p}$, we finally obtain that 
  \begin{equation*}
    \rho_{W^{\alpha}_p}(L(X),L(\tilde{X})) \lesssim_{X,\tilde{X}} \| X - \tilde{X} \|_{W^\alpha_p}.
  \end{equation*}
\end{proof}

\begin{remark}
  While the proofs of Lemma~\ref{lem:regularity of Rf} and \ref{lem:Bound in reconstruction theorem w.r.t. Sobolev model} and Theorem~\ref{thm:continuity of Sobolev rough path lift} contain basically all the necessary ideas to prove the reconstruction theorem for modelled distributions of negative Sobolev regularity and Sobolev models (cf. \cite[Theorem~3.10]{Hairer2014} and \cite[Theorem~2.11]{Liu2016}), we decided not to set up the general theorem for two reasons: to prove Lyons--Victoir extension theorem is currently its main application and other applications might require a different definition of Sobolev type models, cf. Remark~\ref{rmk:discussion Sobolev model}.

  On the other hand, one can of course prove the above reconstruction theorem without using the notations and notions from regularity structure. Here we still decide to illustrate the proof in terms of the language of regularity structure mainly because of a pedagogical purpose: as the construction of rough path lift over H{\"o}lder continuous paths via regularity structure is well--known, we believe it is easy for those readers who are familiar with regularity structure theory to understand our approach and realize the difference between the classical H{\"o}lder case and the current Sobolev setting.
\end{remark}

\begin{remark}
  In general the rough path lift obtained in Theorem~\ref{thm:continuity of Sobolev rough path lift} does not coincide with the rough path lift defined by Riemann--Stieltjes integration \emph{even} in the case of sufficiently regular $\R^d$-valued paths. This is due to the continuity assertion in the Theorem: if the lift coincided on piecewise affine curves (which do have Sobolev regularity) with the standard lift defined by Riemann--Stieltjes integration, then -- by continuity -- it would coincide on limits of such curves. This would yield in particular a rough path lift of H{\"o}lder curves of order $\alpha < 1/2$ continuous with respect to the H{\"o}lder norm and extending classical lifts, which is known to be impossible. However, there exists a class of rough differential equations where the solutions depend only on the driving $\R^d$-valued paths and not on their rough path lifts, see \cite[Section~6]{Lyons2007a}. For these rough differential equations the associated It{\^o}--Lyons map depends, in a meaningful way, continuously only on the $\R^d$-valued driving paths due to continuous lifting map~$L$ provided in Theorem~\ref{thm:continuity of Sobolev rough path lift}.
\end{remark}


\end{document}